\documentclass[12pt,a4paper]{amsart}
\usepackage{a4wide}
\usepackage{amsmath}
\usepackage{amsthm}
\usepackage{amsfonts}
\usepackage{amssymb}
\usepackage{bm}
\usepackage{tikz}
\usepackage{tikz-cd}
\usepackage{float}
\usepackage[all]{xy}
\usepackage{hyperref}
\usepackage{enumerate}
\usepackage{enumitem}
\usepackage[english]{babel}
\usepackage{mathtools}

\newcommand{\op}{\ensuremath{^{\mathrm{op}}}}

\newcommand{\cA}{{\mathcal A}}
\newcommand{\cB}{{\mathcal B}}

\newcommand{\cS}{{\mathcal S}}

\newcommand{\cX}{{\mathcal X}}

\newcommand{\eff}{\operatorname{eff}}

\newcommand{\md}{\operatorname{mod}}
\newcommand{\Md}{\operatorname{Mod}}

\newcommand{\im}{\operatorname{im}}

\newcommand{\Id}{\operatorname{Id}}

\newcommand{\Tr}{\operatorname{Tr}}
\newcommand{\emphbf}[1]{\emph{\textbf{#1}}}
\newcommand{\Hom}{\operatorname{Hom}}

\newcommand{\Ext}{\operatorname{Ext}}
\newcommand{\End}{\operatorname{End}}

\newcommand{\Ker}{\operatorname{Ker}}
\newcommand{\Coker}{\operatorname{Coker}}

\newcommand{\mylabel}[2]{#2\def\@currentlabel{#2}\label{#1}}

\usetikzlibrary{matrix,arrows}
\title{Axiomatizing Subcategories of Abelian Categories}

\date{\today}


\author{Sondre Kvamme}
\address{Sondre Kvamme\\
Department of Mathematics, Uppsala University 
\\ 75106 Uppsala, Sweden
} \email{sondre.kvamme@math.uu.se}

\begin{document}

\newtheorem{Theorem}[equation]{Theorem}
\newtheorem{Lemma}[equation]{Lemma}
\newtheorem{Corollary}[equation]{Corollary}
\newtheorem{Proposition}[equation]{Proposition}
\newtheorem{Conjecture}[equation]{Conjecture}

\theoremstyle{definition}
\newtheorem{Definition}[equation]{Definition}
\newtheorem{Example}[equation]{Example}
\newtheorem{Remark}[equation]{Remark}
\newtheorem{Setting}[equation]{Setting}

\thanks{}


\begin{abstract}
We investigate how to characterize subcategories of abelian categories in terms of intrinsic axioms. In particular, we find axioms which characterize generating cogenerating functorially finite subcategories,  precluster tilting subcategories, and  cluster tilting subcategories of abelian categories. As a consequence we prove that any $d$-abelian category is equivalent to a $d$-cluster tilting subcategory of an abelian category, without any assumption on the categories being projectively generated.
\end{abstract}

\maketitle

\setcounter{tocdepth}{2}
\numberwithin{equation}{section}
\tableofcontents

\section{Introduction}
Higher Auslander-Reiten theory was introduced by Iyama in \cite{Iya07a} and further developed in \cite{Iya07,Iya11}. It has several connections to other areas, for example non-commutative algebraic geometry \cite{HI11a,HIMO14,Iya07}, combinatorics \cite{OT12}, higher category theory \cite{DJW19}, and symplectic geometry \cite{DJL19}. One of the main objects of study are $d$-cluster tilting subcategories of abelian, exact, and triangulated categories. The study of their intrinsic properties, called higher homological algebra, is an active area of research, see for example \cite{Fed19,Fed20,HJV17,JJ17,JJ20,J16,Rei20}. This approach was catalysed by the papers \cite{GKO13} and \cite{Jas16}, where they introduced $d$-abelian, $d$-exact, and $(d+2)$-angulated categories as an axiomatization of $d$-cluster tilting subcategories. In particular, they showed that $d$-cluster tilting subcategories of abelian, exact or triangulated categories are $d$-abelian, $d$-exact, or $(d+2)$-angulated, respectively, and that any projectively generated $d$-abelian category is equivalent to a $d$-cluster tilting subcategory of an abelian category. 

Axiomatizing subcategories of abelian categories is closely related to  characterizing $\Lambda$-modules $M$, where $\Lambda$ is an Artin algebra, in terms of properties of the endomorphism algebra $\Gamma:=\End_{\Lambda}(M)$. In an unpublished manuscript \cite{MT} Morita and Tachikawa showed that $M\mapsto \Gamma$ gives a correspondence between generating and cogenerating modules $M$ and algebras $\Gamma$ with dominant dimension $\geq 2$. Auslander \cite{Aus71} showed that this specializes to the case where $M$ is an additive generator of a module category and $\Gamma$ is an algebra with dominant dimension $\geq 2$ and global dimension $\leq 2$. This is typically called the Auslander-correspondence. It was later extended by Iyama \cite{Iya07} to a bijection between $d$-cluster tilting modules $M$ and algebras $\Gamma$ with dominant dimension $\geq d+1$ and global dimension $\leq d+1$. Recently Iyama and Solberg \cite{IS18} introduced $d$-precluster tilting modules $M$ and showed that the assignment $M\mapsto \Gamma$ gives a bijection to algebras of dominant dimension $\geq d+1$ and selfinjective dimension $\leq d+1$. In all of these cases one characterizes the module $M$ in terms of properties of the category of finitely presented $\Gamma$-modules. One can interpret axiomatization similarly, but where the characterization is in terms of the category  of finitely generated projective $\Gamma$-modules. 

In this paper we continue the idea of axiomatizing  subcategories of abelian categories and study their properties. The following definition clarifies what we mean:

\begin{Definition}\label{Axiomatizing subcategories}
Let $\mathtt{P}$ be a set of axioms of additive categories, and let $\mathtt{S}$ be a class of subcategories of abelian categories. We say that  $\mathtt{P}$ \emphbf{axiomatizes} subcategories in $\mathtt{S}$ if the following hold:
\begin{enumerate}
\item\label{Axiomatizing subcategories:1} If $\cX$ is in $\mathtt{S}$, then $\cX$ satisfies $\mathtt{P}$ as an additive category;
\item\label{Axiomatizing subcategories:2} If $\cX$ satisfies $\mathtt{P}$, then there exists an abelian category $\cA$ and a fully faithful functor $\cX\to \cA$ such that its essential image is in $\mathtt{S}$;
\item\label{Axiomatizing subcategories:3} If $\cA$ and $\cA'$ are abelian categories and $\Phi\colon \cX\to \cA$ and $\Phi'\colon \cX\to \cA'$ are two fully faithful functors such that the essential images of $\Phi$ and $\Phi'$ are in $\mathtt{S}$, then there exists an equivalence $\Psi\colon \cA\xrightarrow{\cong} \cA'$ and a natural isomorphism $\Psi \circ \Phi\cong \Phi'$.
\end{enumerate}
\end{Definition}

Part \ref{Axiomatizing subcategories:1} and \ref{Axiomatizing subcategories:2} tells us that an additive category $\cX$ satisfies $\mathtt{P}$ if and only if it is equivalent to a subcategory in $\mathtt{S}$. Part \ref{Axiomatizing subcategories:3} tells us that the ambient abelian category of a subcategory in $\mathtt{S}$ must be unique up to equivalence. 

We prove the following theorem, which gives examples of axioms $\mathtt{P}$ and classes of subcategories $\mathtt{S}$. See the end of the introduction for the list of axioms.
\begin{Theorem}\label{Main theorem}
Let $\mathtt{P}$ be a set of axioms for additive categories, and let $\mathtt{S}$ be a class of subcategories of abelian categories. We have that $\mathtt{P}$ axiomatizes subcategories in $\mathtt{S}$ in the following cases:
\begin{enumerate}
\item\label{Main theorem:1} $\mathtt{P}$ consists of the axioms (A1), (A1)\op, (A2), (A2)\op, (A3), and (A3)\op , and $\mathtt{S}$ is the class of generating cogenerating functorially finite subcategories;
\item\label{Main theorem:2} $\mathtt{P}$ consists of the axioms (A1), (A1)\op, (A2), (A2)\op, (A3), and (A3)\op, and ($d$-Rigid) and $\mathtt{S}$ is the class of generating cogenerating functorially finite subcategories $\cX$ satisfying
\[
\Ext^i_{\cA}(X,X')=0 \text{ for all }X,X'\in \cX \text{ and }0<i<d
\]
where $\cA$ is the ambient abelian category;
\item\label{Main theorem:3} $\mathtt{P}$ consists of the axioms (A0), (A1), (A1)\op, (A2), (A2)\op, (A3), and (A3)\op, ($d$-Rigid), (A4.$d$), and (A4.$d$)\op , and $\mathtt{S}$ is the class of $d$-precluster tilting subcategories;
\item\label{Main theorem:4} $\mathtt{P}$ consists of the axioms (A0), (A1), (A1)\op, (A2), (A2)\op, (A3), and (A3)\op, ($d$-Rigid), ($d$-Ker), and ($d$-Coker) and $\mathtt{S}$ is the class of $d$-cluster tilting subcategories.
\end{enumerate} 
\end{Theorem}
We even find a class of subcategories axiomatized by (A1) and (A2), see Theorem \ref{Embedding Theorem}. Note that the definition of $d$-precluster tilting subcategories in \cite{IS18} can be reformulated in a way that makes sense for any abelian category, see Theorem \ref{reformulation precluster tilting}, and this reformulated definition is what we use in Theorem \ref{Main theorem} \ref{Main theorem:3}. 

As a corollary of Theorem \ref{Main theorem} \ref{Main theorem:4} we show that any $d$-abelian category is equivalent to a $d$-cluster tilting subcategory, without the assumption of being projectively generated. Since for any $d\geq 1$ there exist examples of $d$-cluster tilting subcategories without any non-zero projective or injective objects, see \cite{JK19}, the result is necessary to complete the axiomatization of $d$-cluster tilting subcategories in terms of $d$-abelian categories.

\begin{Corollary}\label{d-abelian is d-cluster tilting main}
Let $\cX$ be an additive category. The following hold:
\begin{enumerate}
\item\label{d-abelian is d-cluster tilting main:1} $\cX$ is $d$-abelian if and only if it satisfies (A0), (A1), (A1)\op, (A2), (A2)\op, (A3), and (A3)\op, (d-Rigid), (d-Ker), and (d-Coker);
\item\label{d-abelian is d-cluster tilting main:2} If $\cX$ is $d$-abelian, then there exists an abelian category $\cA$ and a fully faithful functor $\cX\to \cA$ such that its essential image is $d$-cluster tilting.\footnote{The author gave a talk about this result in the LMS Northern Regional Meeting and Workshop on Higher Homological Algebra in 2019}
\end{enumerate}
\end{Corollary}

We end the introduction by giving the list of axioms we use:
\begin{enumerate}
\item[(A0)] $\cX$ is idempotent complete;

\item[(A1)] $\cX$ has weak kernels;

\item[(A1)\op] $\cX$ has weak cokernels;

\item[(A2)] Any epimorphism in $\cX$ is a weak cokernel;

\item[(A2)\op] Any monomorphism in $\cX$ is a weak kernel;

\item[(A3)] Consider the following diagram
\begin{equation*}
\begin{tikzcd}
X_2 \arrow{rr}{f} \arrow{rd}{l} & & X_1 \arrow{r}{g} & X_0 \\
 & \arrow{ru}{h}X_2' & 
\end{tikzcd}
\end{equation*}
where $f$ is an arbitrary morphism in $\cX$, where $g$ is a weak cokernel of $f$, where $h$ is a weak kernel of $g$, and where $l$ is an induced map satisfying $h\circ l=f$. Then for any weak kernel $k\colon X_3'\to X_2'$ of $h$ the map $\begin{bmatrix}l & k\end{bmatrix}\colon X_2\oplus X_3'\to X_2'$ is an epimorphism;

\item[(A3)\op] Consider the following diagram
\begin{equation*}
\begin{tikzcd}
X_0 \arrow{r}{g} & X_1 \arrow{rr}{f} \arrow{rd}{h} & & X_2 \\
& & \arrow{ru}{l}X_2' & 
\end{tikzcd}
\end{equation*}
where $f$ is an arbitrary morphism in $\cX$, where $g$ is a weak kernel of $f$, where $h$ is a weak cokernel of $g$, and where $l$ is a map satisfying $l\circ h=f$. Then for any weak cokernel $k\colon X_2'\to X_3'$ of $h$ the map $\begin{bmatrix}l \\ k\end{bmatrix}\colon X_2'\to X_2\oplus X_3'$ is a monomorphism;

\item[(A4.$d$)] Let
\[
X_{d+1}\xrightarrow{f_{d+1}}X_d\xrightarrow{f_d}\cdots \xrightarrow{f_3}X_2\xrightarrow{f_2}X_1\xrightarrow{f_1}X_0\xrightarrow{f_0}X_{-1}
\]
be a sequence with $f_{i+1}$ a weak kernel of $f_i$ for all $0\leq i\leq d$. Then $f_{d+1}$ is a weak cokernel;

\item[(A4.$d$)\op] Let
\[
X_{d+1}\xrightarrow{f_{d+1}}X_d\xrightarrow{f_d}\cdots \xrightarrow{f_3}X_2\xrightarrow{f_2}X_1\xrightarrow{f_1}X_0\xrightarrow{f_0}X_{-1}
\]
be a sequence with $f_{i}$ a weak cokernel of $f_{i+1}$ for all $0\leq i\leq d$. Then $f_{0}$ is a weak kernel;

\item[($d$-Rigid)] For all epimorphism $f_1\colon X_1\to X_0$ in $\cX$ there exists a sequence
\[
X_{d+1}\xrightarrow{f_{d+1}}X_d\xrightarrow{f_d}\cdots \xrightarrow{f_3}X_2\xrightarrow{f_2}X_1\xrightarrow{f_1}X_0
\]
with $f_{i+1}$ a weak kernel of $f_i$ and $f_i$ a weak cokernel of $f_{i+1}$ for all $1\leq i\leq d$;

\item[($d$-Ker)] Any map in $\cX$ has a $d$-kernel;

\item[($d$-Coker)] Any map in $\cX$ has a $d$-cokernel.
\end{enumerate}

Note that the axiom ($d$-Rigid) is self-dual under the assumption of axioms (A1), (A1)\op, (A2), (A2)\op, (A3), and (A3)\op, since these axioms together axiomatizes generating cogenerating functorially finite $d$-rigid subcategories by Theorem \ref{Main theorem} \ref{Main theorem:2}, and the definition of such subcategories are self-dual. 

\subsection{Conventions}
All categories are assumed to be additive, i.e. enriched over abelian groups and admitting finite direct sums. For an additive category $\cX$ we let $\cX(X,X')$ denote the set of morphism between two objects $X,X'\in \cX$ and $\Hom_{\cX}(F_1,F_2)$ the set of natural transformations between two additive functors $F_1,F_2\colon \cX\op\to \operatorname{Ab}$. A subcategory $\cX$ of an abelian category $\cA$ is called generating (resp cogenerating) if for any object $A\in \cA$ there exists an epimorphism $X\to A$ (resp a monomorphism $A\to X$) with $X\in \cX$.

\subsection{Acknowledgement}
Corollary \ref{d-abelian is d-cluster tilting main} \ref{d-abelian is d-cluster tilting main:2} is proved independently by Ramin Ebrahimi and Alireza Nasr-Isfahani in \cite{RA20}. The author would like to thank the anonymous referee for helpful suggestions which has improved the readability of the paper.

\section{Serre subcategories}
In this section we recall the localization of an abelian category by a Serre subcategory. Let $\cA$ be an abelian category. A subcategory $\cS$ of $\cA$ is called a \emph{Serre subcategory} if for any exact sequence in $\cA$
\[
0\to A_1\to A_2\to A_3\to 0
\]
we have that $A_1\in \cS$ and $A_3\in \cS$ if and only if $A_2\in \cS$. If $\cS$ be a Serre subcategory of $\cA$, category $\cA/\cS$ is defined to be the localization of $\cA$ by the class of morphisms $f\colon X\to X'$ satisfying 
\[
\Ker f\in \cS\quad \text{and} \quad \Coker f \in \cS.
\]
Note that the objects in $\cA/\cS$ are the same as the objects in $\cA$. We need the following results for this localization, which follows from  Proposition 1 and Lemma 2 in Chapter 3 in \cite{Gab62}.

\begin{Theorem}\label{Serre subcategory}
Let $\cS$ be a Serre subcategory of $\cA$, and let $q\colon \cA\to \cA/\cS$ denote the canonical functor to the localization. The following hold:
\begin{enumerate}
\item  $\cA/\cS$ is an abelian category;
\item  $q$ is an exact functor;
\item  For a morphism $f\colon A_1\to A_2$ in $\cA$, we have that $q(f)=0$ if and only if $\im f \in \cS$. 
\end{enumerate} 
\end{Theorem}

\begin{Remark}
Note that the class of morphisms $f$ with $\Ker f\in \cS$ and $\Coker f\in \cS$ forms a multiplicative system in $\cA$, see \cite[Exercise 10.3.2 (1)]{Wei94}. In particular, the morphisms in $\cA/\cS$ can be described using a calculus of fractions, see \cite[Chapter 10.3]{Wei94}.
\end{Remark}

\section{Weak kernels and weak cokernels}
In this section we recall the definition of weak kernels and cokernels and their basic properties. Let $\cX$ be an additive category.  We denote the category of additive functors from $\cX\op$ to $\operatorname{Ab}$ by $\Md \cX$. Note that the Yoneda embedding gives a fully faithful functor
\begin{align*}
& \cX\to \Md \cX \quad \quad X\mapsto \cX(-,X) 
\end{align*}
A functor $F\colon \cX\op\to \operatorname{Ab}$ is called \emph{finitely presented} if there exists an exact sequence
\[
\cX(-,X_1)\to \cX(-,X_0)\to F\to 0
\]
in $\Md \cX$, and the subcategory of finitely presented functors is denoted by $\md \cX$. We have that $\md \cX$ is closed under cokernels in $\Md \cX$. Let $g\colon X''\to X'$ and $f\colon X'\to X$  be two composable morphisms in $\cX$. We say that $g$ is a \emph{weak kernel} of $f$ if 
\[
\cX(Y,X'')\xrightarrow{g\circ -}\cX(Y,X')\xrightarrow{f\circ -}\cX(Y,X)
\]
is an exact sequence of abelian groups for all $Y\in \cX$. Dually, we say that $f$ is a \emph{weak cokernel} of $g$ if 
\[
\cX(X,Y)\xrightarrow{-\circ f}\cX(X',Y)\xrightarrow{-\circ g}\cX(X'',Y)
\]
is an exact sequence of abelian groups for all $Y\in \cX$. The category $\cX$ has \emph{weak kernels} or \emph{weak cokernels} if any morphism in $\cX$ has a weak kernel or weak cokernel, respectively. In the following theorem we relate these notions to $\md \cX$.

\begin{Theorem}[Theorem 1.4 in \cite{Fre66}]\label{Freyd}
Let $\cX$ be an additive category. Then $\cX$ has weak kernels if and only if $\md \cX$ is abelian. 
\end{Theorem}

We call a morphism $f$ a weak kernel or weak cokernel if there exists a morphism $g$ such that $f$ is the weak kernel or weak cokernel of $g$, respectively. The following result is known to the experts. We give a proof for the readers convenience.

\begin{Lemma}\label{Weak kernels and cokernels}
Let $f\colon X\to X'$ be a morphism in $\cX$. The following hold:
\begin{enumerate}
\item\label{Weak kernels and cokernels:1} If $f$ is a weak kernel and admits a weak cokernel, then it is a weak kernel of its weak cokernel;
\item\label{Weak kernels and cokernels:2} If $f$ is a weak cokernel and admits a weak kernel, then it is a weak cokernel of its weak kernel.
\end{enumerate}
\end{Lemma}

\begin{proof}
We prove \ref{Weak kernels and cokernels:1}, \ref{Weak kernels and cokernels:2} is proved dually. Assume $f$ is a weak kernel of $g\colon X'\to X''$, and let $h\colon X'\to \tilde{X}$ be a weak cokernel of $f$. Since $g\circ f=0$ and $h$ is a weak cokernel of $f$, there exists a morphism $k\colon \tilde{X}\to X''$ such that $k\circ h=g$. Hence, if a morphism $l\colon \tilde{X'}\to X'$ satisfies $h\circ l=0$, then $g\circ l=k\circ h\circ l=0$, so $l$ factors through $f$ since $f$ is a weak kernel of $g$. This implies that $f$ is a weak kernel of $h$.
\end{proof}

Finally, we recall the definition of contravariantly and covariantly finite subcategories. Assume $\cX$ is an additive subcategory of an abelian category $\cA$. A morphism $X\xrightarrow{f} A$ in $\cA$ with $X\in \cX$ is called a \emph{right} $\cX$-\emph{approximation} of $A$ if any map $X'\to A$ with $X'\in \cX$ factors through $f$. Dually, a morphism $g\colon A\to X$ with $X\in \cX$ is a \emph{left} $\cX$-\emph{approximation} of $A$ if it is a right $\cX\op$-approximation of $A$ in $\cA\op$.  We say that $\cX$ is \emph{contravariantly finite} (resp \emph{covariantly finite}) if any object $A$ in $\cA$ admits a right (resp left) $\cX$-approximation. We say that $\cX$ is \emph{functorially finite} if it is both contravariantly finite and covariantly finite. If $\cX$ is a contravariantly finite subcategory, then it has weak kernels. In fact, if $f\colon X'\to X$ is a morphism in $\cX$ and  $X''\to \Ker f$ is a right $\cX$-approximation, then the composite $X''\to \Ker f\to X'$ is a weak kernel of $f$. Similarly, any covariantly finite subcategory has weak cokernels, which are constructed in the dual way.


\section{Embeddings into abelian categories}
In this section we compare intrinsic axioms of additive categories with properties of subcategories of abelian categories. For an additive category $\cX$, the intrinsic axioms we consider are:

\begin{enumerate}[label=(A\arabic*)]
\item\label{A1} $\cX$ has weak kernels;
\item\label{A2} Any epimorphism in $\cX$ is a weak cokernel.
\end{enumerate}

For an abelian category $\cA$ and a full subcategory $\cX$ of $\cA$, the properties we consider are:
\begin{enumerate}[label=(B\arabic*)]
\item\label{B1} $\cX$ is a generating subcategory of $\cA$;
\item\label{B2} If $A\in \cA$ satisfies $\cA(A,X)=0$ for all $X\in \cX$, then $A=0$;
\item\label{B3} Any $A\in \Omega^2_{\cX}(\cA)$ admits a right $\cX$-approximation.
\end{enumerate}
Here $\Omega^n_{\cX}(\cA)$ denotes the subcategory of $\cA$ consisting of all objects $A$ for which there exists an exact sequence
\[
0\to A\to X_1\to \cdots \to X_n
\]
where $X_i\in \cX$ for all $1\leq i\leq n$. Our main goal is to prove the following theorem:

\begin{Theorem}\label{Embedding Theorem}
Assume $\mathsf{P}$ consists of the axioms \ref{A1} and \ref{A2}. Then $\mathsf{P}$ axiomatizes subcategories of abelian categories satisfying \ref{B1}, \ref{B2} and \ref{B3}. 
\end{Theorem}

We can show one part of the theorem immediately.

\begin{Lemma}\label{B1,B2,B3 implies A1,A2}
Assume $\cX\subseteq \cA$ is a full subcategory of an abelian category $\cA$ satisfying \ref{B1}, \ref{B2} and \ref{B3}. The following hold: 
\begin{enumerate}
\item\label{B1,B2,B3 implies A1,A2:1} The inclusion functor $\cX\to \cA$ preserves epimorphisms and sends a sequence $X_2\xrightarrow{f}X_1\xrightarrow{g}X_0$ in $\cX$ with $f$ a weak kernel of $g$ to an exact sequence in $\cA$;
\item\label{B1,B2,B3 implies A1,A2:2} $\cX$ satisfies \ref{A1} and \ref{A2} as an additive category.
\end{enumerate}
\end{Lemma}

\begin{proof}
Assume $\cX\subseteq \cA$ satisfies \ref{B1}, \ref{B2} and \ref{B3}.  Let $X_1\xrightarrow{g}X_0$ be an arbitrary morphism in $\cX$. Then $\Ker g\in \Omega^2_{\cX}(\cA)$, and hence there exists a right $\cX$-approximation $X_2\to \Ker g$, which is surjective since $\cX$ is a generating subcategory of $\cA$. It follows that the composite $X_2\to \Ker g\to X_1$ is a weak kernel of $g$. This proves \ref{A1} and shows that there exists a weak kernel $X_2\to X_1$ of $g$ such that $X_2\to X_1\xrightarrow{g}X_0$ is exact in $\mathcal{A}$. But then this property must hold for any weak kernel of $g$, since $X_2\to X_1$ must factor through any such map. 

Now assume $g$ is an epimorphism, and let $\Coker g$ be the cokernel of $g$ in $\cA$. Applying $\cA(-,X)$ with $X\in \cX$ to the exact sequence
\[
X_1\xrightarrow{g}X_0\to \Coker g\to 0
\]
gives an exact sequence 
\[
0\to \cA(\Coker g,X)\to \cA(X_0,X)\xrightarrow{-\circ g} \cA(X_1,X)
\]
If $X\in \cX$, then since $g$ is an epimorphism in $\cX$ it follows that $-\circ g$ is a monomorphism, and hence $\cA(\Coker g,X)=0$. By \ref{B2} it follows that $\Coker g=0$, so $g$ is an epimorphism in $\cA$. Hence the inclusion $\cX\to \cA$ preserves epimorphisms. Finally, since $X_2\to X_1\xrightarrow{g}X_0\to 0$ is exact, $g$ is a cokernel of $X_2\to X_1$, which proves \ref{A2}.  
\end{proof}

Now assume $\cX$ is an additive category satisfying \ref{A1} and \ref{A2}. Since $\cX$ has weak kernels, the category of finitely presented functors $\md\cX$ is abelian by Theorem \ref{Freyd}. Let $\eff\cX$ denote the subcategory of $\md\cX$ consisting of all functors $F$ for which there exists an exact sequence
\[
\cX(-,X_1)\xrightarrow{f\circ -}\cX(-,X_0)\to F\to 0
\]
where $f\colon X_1\to X_0$ is an epimorphism in $\cX$. We show that $\eff \cX$ is a Serre subcategory of $\md \cX$.

\begin{Proposition}\label{effaceable functors}
Let $\cX$ be an additive category satisfying \ref{A1} and \ref{A2}. The following hold:
\begin{enumerate}
\item\label{effaceable functors:1} If 
\[
\cX(-,X'_1)\xrightarrow{f'\circ -}\cX(-,X'_0)\to F\to 0
\]
is exact with $F\in \eff \cX$, then $f'\colon X_1'\to X_0'$ is an epimorphism in $\cX$;
\item\label{effaceable functors:2} $\eff \cX=\{F\in \md\cX\mid \Hom_{\cX}(F,\cX(-,X))=0 \text{ for all } X\in \cX\}$;
\item\label{effaceable functors:3} If $F\in \eff \cX$ then $\Ext^1_{\md\cX}(F,\cX(-,X))=0$ for all $X\in \cX$;
\item\label{effaceable functors:4} $\eff \cX$ is a Serre subcategory of $\md \cX$.
\end{enumerate}
\end{Proposition}

\begin{proof}
If $F\in \eff \cX$, then there exists an exact sequence 
\[
\cX(-,X_1)\xrightarrow{f\circ -}\cX(-,X_0)\to F\to 0
\]
where $f\colon X_1\to X_0$ is an epimorphism in $\cX$. Applying $\Hom_{\cX}(-,\cX(-,X))$ gives the exact sequence
\[
0\to \Hom_{\cX}(F,\cX(-,X))\to \cX(X_0,X)\xrightarrow{-\circ f}\cX(X_1,X)
\]
Since $f$ is an epimorphism in $\cX$, the map $-\circ f$ is a monomorphism, and hence 
\[
\Hom_{\cX}(F,\cX(-,X))=0.
\]
Conversely, assume $\Hom_{\cX}(F,\cX(-,X))=0$ for all $X\in \cX$. Applying $\Hom_{\cX}(-,\cX(-,X))$ to an exact sequence of the form 
\[
\cX(-,X'_1)\xrightarrow{f'\circ -}\cX(-,X'_0)\to F\to 0
\]
gives an exact sequence
\[
0\to \Hom_{\cX}(F,\cX(-,X))\to \cX(X'_0,X)\xrightarrow{-\circ f'}\cX(X'_1,X)
\]
Since $\Hom_{\cX}(F,\cX(-,X))=0$, the map $-\circ f'$ is a monomorphism, and hence $f'\colon X_1'\to X_0$ is an epimorphism. This proves \ref{effaceable functors:1} and \ref{effaceable functors:2}. For \ref{effaceable functors:3}, assume again we have an exact sequence $\cX(-,X_1)\xrightarrow{f\circ -}\cX(-,X_0)\to F\to 0$ where $f\colon X_1\to X_0$ is an epimorphism in $\cX$. Let $g\colon X_2\to X_1$ be a weak kernel of $f$. Then 
\[
\cX(-,X_2)\xrightarrow{g\circ -}\cX(-,X_1)\xrightarrow{f\circ -}\cX(-,X_0)\to F\to 0
\]
is exact. Applying $\Hom_{\cX}(-,\cX(-,X))$ with $X\in \cX$ gives a complex
\[
0\to \Hom_{\cX}(F,\cX(-,X))\to \cX(X_0,X)\xrightarrow{-\circ f}\cX(X_1,X)\xrightarrow{-\circ g}\cX(X_2,X)
\]
Since $f$ is an epimorphism in $\cX$, it is a weak cokernel by \ref{A2}. By Lemma \ref{Weak kernels and cokernels} \ref{Weak kernels and cokernels:2} we know that $f$ must be a weak cokernel of $g$. Therefore the sequence 
\[
\cX(X_0,X)\xrightarrow{-\circ f}\cX(X_1,X)\xrightarrow{-\circ g}\cX(X_2,X)
\]
 must be exact. This shows that $\Ext^1_{\md\cX}(F,\cX(-,X))=0$.

Finally, we show that $\eff \cX$ is a Serre subcategory. Let 
\[
0\to F_1\to F_2\to F_3\to 0
\]
be an exact sequence in $\md \cX$. Applying $\Hom_{\cX}(-,\cX(-,X))$, we get an exact sequence
\begin{multline*}
0\to \Hom_{\cX}(F_3,\cX(-,X))\to \Hom_{\cX}(F_2,\cX(-,X))\to \Hom_{\cX}(F_1,\cX(-,X)) \\
\to \Ext^1_{\md\cX}(F_3,\cX(-,X))\to \Ext^1_{\md\cX}(F_2,\cX(-,X))
\end{multline*}
Now if $F_1\in\eff\cX$ and $F_3\in \eff \cX$, then $\Hom_{\cX}(F_1,\cX(-,X))=0$ and $\Hom_{\cX}(F_3,\cX(-,X))=0$ for all $X\in \cX$ by part \ref{effaceable functors:2} of the theorem. Therefore, $\Hom_{\cX}(F_2,\cX(-,X))=0$ for all $X\in \cX$, so $F_2\in \eff \cX$. Conversely, assume $F_2\in \eff \cX$. Since $\Hom_{\cX}(F_2,\cX(-,X))=0$ for all $X\in \cX$ by \ref{effaceable functors:2} and $\Ext^1_{\md\cX}(F_2,\cX(-,X)=0$ by \ref{effaceable functors:3}, the exact sequence above correspondence to the exact sequence
\[
0\to \Hom_{\cX}(F_3,\cX(-,X))\to 0 \to \Hom_{\cX}(F_1,\cX(-,X))\to \Ext^1_{\md\cX}(F_3,\cX(-,X))\to 0
\]
Hence $\Hom_{\cX}(F_3,\cX(-,X))=0$ and therefore $\Ext^1_{\md\cX}(F_3,\cX(-,X))=0$ by \ref{effaceable functors:3}. This implies that $\Hom_{\cX}(F_1,\cX(-,X))=0$ for all $X\in \cX$, and so $F_1\in \eff\cX$ and $F_3\in \eff \cX$.
\end{proof}

Since $\eff \cX$ is a Serre subcategory of $\md \cX$, the localization $\md \cX/\eff \cX$ is an abelian category, see Theorem \ref{Serre subcategory}.  We let $q\colon \md \cX\to \md \cX/\eff \cX$ denote the localization functor. By abuse of notation, we denote an object in $\md \cX$ and its image in $\md \cX/\eff \cX$ by the same letter. Consider the functor $\Phi\colon \cX\to \md \cX/\eff \cX$ given by the composite
\[
 \cX \to \md \cX \xrightarrow{q} \md \cX/\eff \cX 
\]
where $\cX\to \md \cX$ is the Yoneda functor. It plays the role of the fully faithful functor in Definition \ref{Axiomatizing subcategories} \ref{Axiomatizing subcategories:2}. To show this, we need the following lemma.

\begin{Lemma}\label{Hom sending maps to isos}
Let $\cX$ be an additive category satisfying \ref{A1} and \ref{A2}. Let $\phi\colon F_0\to F_1$ be a morphism in $\md \cX$ with $\Ker \phi\in \eff \cX$ and $\Coker \phi \in \eff \cX$. Then 
\[
\Hom_{\cX}(F_1,\cX(-,X))\xrightarrow{-\circ \phi}\Hom_{\cX}(F_0,\cX(-,X))
\]
is an isomorphism for all $X\in \cX$.
\end{Lemma}

\begin{proof}
Applying $\Hom_{\cX}(-,\cX(-,X))$ to the exact sequence
\[
0\to \im \phi \to F_1\to \Coker \phi \to 0
\]
gives an exact sequence
\begin{multline*}
0\to \Hom_{\cX}(\Coker \phi,\cX(-,X))\to \Hom_{\cX}(F_1,\cX(-,X)) \\
\to \Hom_{\cX}(\im \phi,\cX(-,X))\to \Ext^1_{\md\cX}(\Coker \phi,\cX(-,X))
\end{multline*}
Since $\Coker \phi\in \eff \cX$, it follows that 
\[
\Hom_{\cX}(\Coker \phi,\cX(-,X))=0 \quad \text{and} \quad \Ext^1_{\md\cX}(\Coker \phi,\cX(-,X))=0
\]
by Theorem \ref{effaceable functors}. Hence the map 
\[
\Hom_{\cX}(F_1,\cX(-,X)) \to \Hom_{\cX}(\im \phi,\cX(-,X))
\]
 is an isomorphism. Similarly, applying $\Hom_{\cX}(-,\cX(-,X))$ to the exact sequene
\[
0\to \Ker \phi \to F_0\to \im \phi\to 0
\]
and using that $\Hom_{\cX}(\Ker \phi,\cX(-,X))=0$, it follows that the map 
\[
\Hom_{\cX}(\im \phi,\cX(-,X)) \to \Hom_{\cX}(F_0,\cX(-,X))
\]
is an isomorphism. Since $\Hom_{\cX}(F_1,\cX(-,X))\xrightarrow{-\circ \phi}\Hom_{\cX}(F_0,\cX(-,X))$ is equal to the composite
\[
\Hom_{\cX}(F_1,\cX(-,X)) \xrightarrow{\cong} \Hom_{\cX}(\im \phi,\cX(-,X)) \xrightarrow{\cong} \Hom_{\cX}(F_0,\cX(-,X))
\]
of two isomorphisms, the claim follows.
\end{proof}

To simplify notation we let $\cA:= \md \cX/\eff \cX$. The localization functor $q\colon \md \cX\to \cA$ induces a map
\[
q\colon \Hom_{\cX}(F,\cX(-,X))\to \cA(F,\cX(-,X)) \quad f\mapsto q(f)
\]
for all $X\in \cX$ and $F\in \md \cX$, which we also denote by $q$ by abuse of notation. We use the previous lemma to show that it is an isomorphism.

\begin{Lemma}\label{Fully faithful}
Let $\cX$ be an additive category satisfying \ref{A1} and \ref{A2}. The map
\[
q\colon \Hom_{\cX}(F,\cX(-,X))\to \cA(F,\cX(-,X)) \quad f\mapsto q(f)
\]
is an isomorphism for all $X\in \cX$ and $F\in \md \cX$.
\end{Lemma}

\begin{proof}
The functor
\[
\Hom_{\cX}(-,\cX(-,X))\colon (\md \cX)\op\to \operatorname{Ab}
\]
induces a well-defined functor
\[
\Hom_{\cX}(-,\cX(-,X))\colon \cA\op\to \operatorname{Ab}
\]
by Lemma \ref{Hom sending maps to isos} and the universal property of the localization. Furthermore, $q$ induces a natural transformation $q\colon  \Hom_{\cX}(-,\cX(-,X))\to \cA(-,\cX(-,X))$ of functors $\cA\op\to \operatorname{Ab}$. Also, by Yoneda's Lemma the element $1_{\cX(-,X)}\in \Hom_{\cX}(\cX(-,X),\cX(-,X))$ corresponds to a natural transformation
\[
\mu\colon \cA(-,\cX(-,X))\to \Hom_{\cX}(-,\cX(-,X))
\]
of functors $\cA\op\to \operatorname{Ab}$. Since the composite
\[
\cA(-,\cX(-,X))\xrightarrow{\mu} \Hom_{\cX}(-,\cX(-,X))\xrightarrow{q} \cA(-,\cX(-,X))
\]
sends $1_{\cX(-,X)}\in \cA(\cX(-,X),\cX(-,X))$ to itself, it must be the identity map. Furthermore, the composite 
\[
\mu\circ q\colon \Hom_{\cX}(-,\cX(-,X))\to \Hom_{\cX}(-,\cX(-,X))
\]
must also be a natural transformation when we consider $\Hom_{\cX}(-,\cX(-,X))$ as a functor $(\md \cX)\op\to \operatorname{Ab}$. Since it sends $1_{\cX(-,X)}\in \Hom_{\cX}(\cX(-,X),\cX(-,X))$ to itself, it must be the identity map. This shows that $q$ is an isomorphism.
\end{proof}

Now we are ready to show that any additive category satisfying \ref{A1} and \ref{A2} is equivalent to a subcategory satisfying \ref{B1}, \ref{B2} and \ref{B3}. Recall that $\Phi\colon \cX\to \md \cX/\eff \cX$ is the functor given by the composite of the Yoneda embedding $\cX\to \md \cX$ and the localization functor $q\colon \md \cX\to \md \cX/\eff \cX$. For simplicity, we also denote the essential image of $\Phi$ by $\cX$ in the proof. 

\begin{Proposition}\label{essential image satisfy B1-B3}
Let $\cX$ be an additive category satisfying \ref{A1} and \ref{A2}. Then the functor $\Phi\colon \cX\to \md \cX/\eff \cX$ is fully faithful, and its essential image satisfies \ref{B1}, \ref{B2} and \ref{B3}.
\end{Proposition}

\begin{proof}
By Lemma \ref{Fully faithful} the canonical map
\[
\Hom_{\cX}(\cX(-,X'),\cX(-,X))\to \cA(\cX(-,X'),\cX(-,X))
\]
is an isomorphism for all $X,X'\in \cX$, where $\cA=\md \cX/\eff \cX$. Since the Yoneda embedding is fully faithful, the functor $\Phi$ must therefore be fully faithful. Now let $F\in \md \cX$ be arbitrary, and choose an epimorphism $p\colon\cX(-,X)\to F$ in $\md \cX$. Since $q$ is exact, $q(p)$ is an epimorphism in $\cA$. This shows \ref{B1}, i.e. that $\cX$ is generating. Next, assume $\cA(F,\cX(-,X))=0$ for all $X\in \cX$. Then $\Hom_{\cX}(F,\cX(-,X))=0$ for all $X\in \cX$ by Lemma \ref{Fully faithful}, and therefore $F\in \eff \cX$ by Lemma \ref{effaceable functors}. It follows that $F\cong 0$ considered as an object in $\cA$, which shows \ref{B2}. Finally, to prove \ref{B3}, assume $F\in \Omega^2_{\cX}(\cA)$. Choose an exact sequence
\[
0\to F\xrightarrow{\phi} \cX(-,X_1)\xrightarrow{\psi}\cX(-,X_0)
\]
in $\cA$. By Lemma \ref{Fully faithful} the maps $\phi$ and $\psi$ can be lifted to morphisms in $\md \cX$ (which we denote by the same name). In particular, we have a morphism $g\colon X_1\to X_0$ in $\cX$ such that $\psi=g\circ-$. Let $f$ be the weak kernel of $g$ in $\cX$. Then the sequence of functors
\[
\cX(-,X_2)\xrightarrow{f\circ -}\cX(-,X_1)\xrightarrow{g\circ -}\cX(-,X_0)
\]
is exact in $\md \cX$. Since $q$ is an exact functor, the sequence is also exact in $\cA$. Furthermore, since $F$ is the kernel of $g\circ -$ in $\cA$, there exists an epimorphism $\xi\colon \cX(-,X_2)\to F$ in $\cA$ such that $f\circ -=\phi\circ \xi$. Now let $\kappa\colon \cX(-,X)\to F$ be an arbitrary morphism in $\cA$ with $X\in \cX$. Then the composite $\phi\circ \kappa\colon \cX(-,X)\to \cX(-,X_1)$ can be written as $h\circ -\colon \cX(-,X)\to \cX(-,X_1)$ for some morphism $h\colon X\to X_1$ in $\cX$, by Lemma \ref{Fully faithful}. Since $g\circ h=0$, and $f$ is a weak kernel of $g$, it follows that $h$ factors through $f$. Therefore the map $\phi \circ \kappa$ factors through $f\circ -\colon \cX(-,X_2)\to \cX(-,X_1)$. Since $\phi$ is a monomorphism, it follows that $\kappa$ factors through $\xi$. Hence $\xi$ is a right $\cX$-approximation, which proves the claim.
\end{proof}

Our next goal is to prove the uniqueness of the ambient abelian category as required in Definition \ref{Axiomatizing subcategories} \ref{Axiomatizing subcategories:3}. To this end, we need the following results describing a universal property of $\md \cX/\eff \cX$.

\begin{Lemma}\label{universal property quotient}
Let $\cX$ be an additive category satisfying \ref{A1} and \ref{A2}. The following hold:
\begin{enumerate}
\item\label{universal property quotient:1} The functor $\Phi\colon \cX\to \md \cX/\eff \cX$ preserves epimorphisms and sends a sequence $X_2\xrightarrow{f}X_1\xrightarrow{g}X_0$ in $\cX$ with $f$ a weak kernel of $g$ to an exact sequence
\[
\Phi(X_2)\xrightarrow{\Phi(f)}\Phi(X_1)\xrightarrow{\Phi(g)}\Phi(X_0)
\]
in $\md \cX/\eff \cX$;
\item\label{universal property quotient:2} Let $\cB$ be an abelian category and let $\Psi\colon \cX\to \cB$ be an additive functor which preserves epimorphisms and sends a sequence $X_2\xrightarrow{f}X_1\xrightarrow{g}X_0$ in $\cX$ with $f$ a weak kernel of $g$ to an exact sequence
\[
\Psi(X_2)\xrightarrow{\Psi(f)}\Psi(X_1)\xrightarrow{\Psi(g)}\Psi(X_0)
\]
in $\cB$. Then there exists an exact functor $\md \cX/\eff \cX\to \cB$ extending $\Psi$, which is unique up to natural isomorphism.
\end{enumerate}
\end{Lemma}

\begin{proof}
Since the essential image of $\Phi$ satisfies \ref{B1}, \ref{B2} and \ref{B3} by Proposition \ref{essential image satisfy B1-B3}, part \ref{universal property quotient:1} follows from Lemma \ref{B1,B2,B3 implies A1,A2}. For part \ref{universal property quotient:2}, let $\Psi$ be a functor as in the lemma. Then there exists a right exact functor $\tilde{\Psi}\colon \md\cX\to \cB$ extending $\Psi$, see Property 2.1 in \cite{Kra98}. Since $\Psi$ sends weak kernels to exact sequences, it follows that $\tilde{\Psi}$ is an exact functor by Lemma 2.5 in \cite{Kra98}. Now let $F\in \md \cX$ be arbitrary, and choose an exact sequence $\cX(-,X_1)\xrightarrow{f\circ -}\cX(-,X_0)\to F\to 0$. Applying $\tilde{\Psi}$ gives an exact sequence
\[
\Psi(X_1)\xrightarrow{\Psi(f)}\Psi(X_0)\to \tilde{\Psi}(F)\to 0
\]
in $\cB$. If $F\in \eff \cX$, then $f\colon X_1\to X_0$ is an epimorphism in $\cX$ by Proposition \ref{effaceable functors} \ref{effaceable functors:1}. Therefore $\Psi(f)$ is an epimorphism in $\cB$, which implies that $\tilde{\Psi}(F)=0$. This shows that $\eff \cX \subseteq \Ker \tilde{\Psi}$. Therefore, by \cite[Corollaire 2 and 3 on page 368-369]{Gab62} there exists an exact functor
\[
\overline{\Psi}\colon \md\cX/\eff \cX\to \cB
\]
satisfying $\overline{\Psi}\circ q=\tilde{\Psi}$. The fact that $\tilde{\Psi}$ is unique follows readily from the fact that $\cX$ is generating in $\md \cX/\eff \cX$.
\end{proof}

Now we can show the uniqueness of the ambient abelian category.

\begin{Proposition}\label{uniqueness ambient abelian}
Let $\cA'$ be an abelian category and $\cX$ an additive subcategory of $\cA'$ which satisfies \ref{B1}, \ref{B2} and \ref{B3}. Then there exists an equivalence $\md \cX/\eff \cX\xrightarrow{\cong} \cA'$ unique up to natural isomorphism which makes the diagram
\begin{equation*}
\begin{tikzcd}
\md\cX/\eff \cX \arrow{rr}{\cong} & & \cA' \\
 & \arrow{lu}{\Phi} \arrow{ru}[swap]{\operatorname{inclusion}}\cX & 
\end{tikzcd}
\end{equation*}
commute.
\end{Proposition}

\begin{proof}
By Lemma \ref{B1,B2,B3 implies A1,A2} and Lemma \ref{universal property quotient} \ref{universal property quotient:2} we have exact functors 
\[
\tilde{\Psi}\colon \md \cX\to \cA' \quad \text{and} \quad \overline{\Psi}\colon \md \cX/\eff \cX\to \cA'
\]
extending the inclusion $\cX\subseteq \cA'$. Our goal is to show that $\overline{\Psi}$ is an equivalence. Since $\cX$ is a generating subcategory of $\cA'$, it follows immediately that $\overline{\Psi}$ is dense. We show that $\overline{\Psi}$ is faithful. Let $F\in \md \cX$ be arbitrary and let $\cX(-,X_1)\xrightarrow{f\circ -}\cX(-,X_0)\to F\to 0$ be an exact sequence in $\md \cX$. Applying $\tilde{\Psi}$, we get an exact sequence
\[
X_1\xrightarrow{f}X_0\to \tilde{\Psi}(F)\to 0
\]
in $\cA'$. Hence, if $\tilde{\Psi}(F)\cong 0$, then  $X_1\xrightarrow{f}X_0$ is surjective in $\md \cX/\eff \cX$. Therefore, $f$ must be surjective in $\cX$, so $F\in \eff \cX$. This shows that $\Ker \tilde{\Psi}=\eff \cX$. Now if $\phi$ is a morphism in $\md \cX/\eff \cX$, then $\overline{\Psi}(\phi)=0$ if and only $\overline{\Psi}(\im \phi)=0$ since $\overline{\Psi}$ is exact. Also, since $\Ker \tilde{\Psi}=\eff \cX$, it follows that $\im \phi \in \eff \cX$, which implies that $\phi=0$ in $\md \cX/\eff \cX$. This shows that $\overline{\Psi}$ is faithful.

 It now only remains to show that $\overline{\Psi}$ is full. Let $F,F'\in \md \cX$, and let 
\[
\cX(-,X_1)\xrightarrow{f\circ -}\cX(-,X_0)\xrightarrow{\pi}F\to 0 \quad \text{and} \quad  \cX(-,X'_1)\xrightarrow{f'\circ -}\cX(-,X'_0)\xrightarrow{\pi'}F'\to 0 
\] 
be exact sequences in $\md \cX$. Let $\phi\colon \overline{\Psi}(F)\to \overline{\Psi}(F')$ be a morphism in $\cA'$. We then get a commutative diagram
\begin{equation*}
\begin{tikzcd}[row sep =5ex, column sep =11ex, ampersand replacement =\&]
X_1 \arrow{r}{f} \arrow{d}{} \& X_0 \arrow{r}{\overline{\Psi}(\pi)} \arrow{d}{\begin{bmatrix}
1 \\ 0\end{bmatrix}} \& \overline{\Psi}(F) \arrow{r}{} \arrow{d}{\phi} \& 0 \\
K \arrow{r}{} \& X_0\oplus X_0'\arrow{r}{\begin{bmatrix}\phi\circ \overline{\Psi}(\pi)& \overline{\Psi}(\pi')\end{bmatrix}} \& \overline{\Psi}(F') \arrow{r}{} \& 0 \\
X_1'\arrow{r}{f'} \arrow{u}{} \& X_0'\arrow{r}{\overline{\Psi}(\pi')} \arrow{u}{\begin{bmatrix}0 \\ 1\end{bmatrix}} \& \overline{\Psi}(F') \arrow{r}{} \arrow{u}{1} \& 0
\end{tikzcd}
\end{equation*}
where the rows are exact and $K=\Ker \begin{bmatrix}\phi\circ \overline{\Psi}(\pi)& \overline{\Psi}(\pi')\end{bmatrix}$. Choose an epimorphism $\tilde{X_1}\to K$ in $\cA'$ with $\tilde{X_1}\in \cX$ and set $X_1''=X_1\oplus X_1'\oplus \tilde{X}_1$. Then we can factorize the map $X_1\oplus X_1'\to K$ as a composition $X_1\oplus X_1'\to X_1''\to K$ where $X_1''\to K$ is an epimorphism. Let $\begin{bmatrix}g \\ h\end{bmatrix}\colon X_1''\to X_0\oplus X_0'$ denote the map obtained by composing this epimorphism $X_1''\to K$ with the inclusion $K\to X_0\oplus X_0'$, let 
\[
F''=\Coker (\cX(-,X_1'')\xrightarrow{\begin{bmatrix}g \\ h\end{bmatrix}\circ -}\cX(-,X_0\oplus X_0'))
\]
and let $\pi''\colon \cX(-,X_0\oplus X_0')\to F''$ denote the projection. Then we get morphisms $\phi'\colon F\to F''$ and $\phi''\colon F'\to F''$ making the diagram
\begin{equation*}
\begin{tikzcd}[row sep =5ex, column sep =11ex, ampersand replacement =\&]
\cX(-,X_1) \arrow{r}{f\circ -} \arrow{d}{} \& \cX(-,X_0) \arrow{r}{\pi} \arrow{d}{\begin{bmatrix}
1 \\ 0\end{bmatrix}\circ -} \& F \arrow{r}{} \arrow{d}{\phi'} \& 0 \\
\cX(-,X_1'') \arrow{r}{\begin{bmatrix}g \\ h\end{bmatrix}\circ -} \& \cX(-,X_0\oplus X_0')\arrow{r}{\pi''} \& F'' \arrow{r}{} \& 0 \\
\cX(-,X_1')\arrow{r}{f'\circ -} \arrow{u}{} \& \cX(-,X_0')\arrow{r}{\pi'} \arrow{u}{\begin{bmatrix}0 \\ 1\end{bmatrix}\circ -} \&F' \arrow{r}{} \arrow{u}{\phi''} \& 0
\end{tikzcd}
\end{equation*}
commutative. Note furthermore that $\tilde{\Psi}(\phi'')$ is an isomorphism, hence 
\[
\tilde{\Psi}(\Ker \phi'')\cong 0\cong \tilde{\Psi}(\Coker \phi'')
\]
since $\tilde{\Psi}$ is exact. This implies that $\Ker \phi''\in \eff \cX$ and $\Coker \phi''\in \eff \cX$ since $\Ker \tilde{\Psi}=\eff \cX$. Therefore, $\phi''$ is an isomorphism in $\md \cX/\eff \cX$, and hence admits an inverse $\phi''^{-1}$. It is now clear that 
\[
\phi=\overline{\Psi}(\phi''^{-1})\circ \overline{\Psi}(\phi')=\overline{\Psi}(\phi''^{-1}\circ \phi')
\]
which shows that $\overline{\Psi}$ is full. This proves the claim.
\end{proof}

\begin{proof}[Proof of Theorem \ref{Embedding Theorem}]
This follows from Lemma \ref{B1,B2,B3 implies A1,A2}, Proposition \ref{essential image satisfy B1-B3}, and Proposition \ref{uniqueness ambient abelian}.
\end{proof}
As a consequence we get that the ambient abelian category also satisfies a universal property.

\begin{Corollary}\label{universal property ambient}
Let $\cA'$ an abelian category and let $\cX$ an additive subcategory of $\cA'$ which satisfies \ref{B1}, \ref{B2} and \ref{B3}. The following hold:
\begin{enumerate}
\item\label{universal property ambient:1} The inclusion functor $\cX\to \cA'$ preserves epimorphisms and sends a sequence $X_2\xrightarrow{f}X_1\xrightarrow{g}X_0$ in $\cX$ with $f$ a weak kernel of $g$ to an exact sequence in $\cA'$;
\item\label{universal property ambient:2} Let $\cB$ be an abelian category and let $\Psi\colon \cX\to \cB$ be an additive functor which preserves epimorphisms and sends a sequence $X_2\xrightarrow{f}X_1\xrightarrow{g}X_0$ in $\cX$ with $f$ a weak kernel of $g$ to an exact sequence
\[
\Psi(X_2)\xrightarrow{\Psi(f)}\Psi(X_1)\xrightarrow{\Psi(g)}\Psi(X_0)
\]
in $\cB$. Then there exists an exact functor $\cA'\to \cB$ extending $\Psi$, which is unique up to natural isomorphism.
\end{enumerate}
\end{Corollary}

\begin{proof}
This follows immediately from Lemma \ref{universal property quotient} and Proposition \ref{uniqueness ambient abelian}.
\end{proof}

\section{Functorially finite generating cogenerating subcategories}

Let $\cX$ be an additive category. We would like to find intrinsic axioms on $\cX$ which axiomatizes  functorially finite generating and cogenerating subcategories of abelian categories, as in Definition \ref{Axiomatizing subcategories}. To this end, by Theorem \ref{Embedding Theorem} we know that $\cX$ must satisfy \ref{A1} and \ref{A2} and their duals
\begin{enumerate}[label=(A\arabic*)$^{\operatorname{op}}$]
\item\label{A1op} $\cX$ has weak cokernels;
\item\label{A2op} Any monomorphism in $\cX$ is a weak kernel.
\end{enumerate}
To continue, we investigate the duality functor on $\md \cX$. For each $F\in \md \cX$ choose a projective presentation 
\[
\cX(-,X_1)\xrightarrow{f\circ -}\cX(-,X_0)\to F\to 0
\]
in $\md \cX$. Applying the contravariant Yoneda functor to $f\colon X_1\to X_0$ gives a map \\$-\circ f\colon \cX(X_0,-)\to \cX(X_1,-)$ in $\md \cX\op$. We define $F^*=\Ker(-\circ f)$ and $\Tr F=\Coker (-\circ f)$, so that we have an exact sequence
\[
0\to F^*\to \cX(X_0,-)\xrightarrow{-\circ f}\cX(X_1,-)\to \Tr F\to 0.
\]
in $\md \cX\op$. Dually, for $F'\in \md \cX\op$ we choose a projective presentation 
\[
\cX(X_0',-)\xrightarrow{-\circ f'}\cX(X_1',-)\to F'\to 0
\]
and define $\Tr F'$ and $F'^*$ by the exact sequence
\[
0\to F'^*\to \cX(-,X_1')\xrightarrow{f'\circ -}\cX(-,X_0')\to \Tr F'\to 0.
\]
Note that we have natural isomorphisms
\[
(-)^*\cong\Hom_{\cX}(-,\cX)\colon \md \cX \to \md \cX\op  
\]
and 
\[
(-)^*\cong\Hom_{\cX\op}(-,\cX\op)\colon \md \cX\op \to \md \cX  
\]
which we use to identify these functors, where $\Hom_{\cX}(-,\cX)$ and $\Hom_{\cX\op}(-,\cX\op)$ denote the functors given by $\Hom_{\cX}(-,\cX)(F)=\Hom_{\cX}(F,-)|_{\cX}$ and $\Hom_{\cX\op}(-,\cX\op)(G)=\Hom_{\cX\op}(G,-)|_{\cX\op}$ for $F\in \md\cX$ and $G\in \md\cX\op$, respectively. It follows that the functors $(-)^*\colon \md \cX \to \md \cX\op$ and $(-)^*\colon \md \cX\op \to \md \cX$ form an adjoint pair. The unit and counit are part of exact sequences 
\begin{align}\label{unit of (-)^*}
& 0\to \Ext^2_{\md \cX\op}(\operatorname{Tr}F,\cX\op)\to F\to F^{**}\to \Ext^1_{\md \cX\op}(\operatorname{Tr}F,\cX\op)\to 0 
\end{align}
\begin{align}\label{counit of (-)^*}
 0\to \Ext^2_{\md \cX}(\operatorname{Tr}F,\cX)\to F\to F^{**}\to \Ext^1_{\md\cX}(\operatorname{Tr}F,\cX)\to 0
\end{align}
see for example Proposition 6.3 in \cite{Aus66} in the case $\cX$ is a ring. Now by Lemma \ref{Hom sending maps to isos} the functor $(-)^*$ sends morphisms in $\md \cX$ with kernel and cokernel in $\eff \cX$ to isomorphisms in $\md \cX\op$. Together with the dual statement this implies that there are induced contravariant functors
\[
\md \cX/\eff \cX\to \md \cX\op \quad \text{and} \quad \md \cX\op/\eff \cX\op\to \md \cX
\]
making the following diagrams commute
\begin{equation*}
\begin{tikzcd}
\md\cX \arrow{r}{(-)^*} \arrow{d}{q} & \md \cX\op \\
 \arrow{ru}{} \md \cX/\eff \cX & 
\end{tikzcd} \quad 
\begin{tikzcd}
\md\cX\op \arrow{r}{(-)^*} \arrow{d}{q} & \md \cX \\
 \arrow{ru}{} \md \cX\op/\eff \cX\op & 
\end{tikzcd}
\end{equation*}
Composing with $q$ gives contravariant functors
\begin{align*}
&(-)^*\colon \md \cX/\eff \cX\to \md \cX\op/\eff \cX\op \\
& (-)^*\colon \md \cX\op/\eff \cX\op\to \md \cX/\eff \cX
\end{align*}
which we denote by the same symbol. The natural transformations $\Id \to (-)^{**}$ satisfy the triangular identities in $\md \cX$ and $\md \cX\op$, and hence also in $\md \cX/\eff \cX$ and $\md \cX\op/\eff \cX\op$. This implies that the functors $(-)^*$ still form an adjoint pair as functors between $\md \cX/\eff \cX$ and $\md \cX\op/\eff \cX\op$.

Next we want to find conditions on $\cX$ which ensures that $(-)^*$ induces an equivalence
\[
\md \cX/\eff \cX\xrightarrow{\cong}(\md \cX\op/\eff \cX\op)\op.
\] 
Note first that for $F\in \md\cX$ we have an isomorphism
\[
\Ext^2_{\md \cX}(F,G)\cong \Ext^1_{\md\cX}(F',G)
\]
where $F'$ is a syzygy of $F$, i.e. fits in an exact sequence
\[
0\to F'\to \cX(-,X)\to F\to 0
\] 
Using this, we see that the unit and counit given by \eqref{unit of (-)^*} and \eqref{counit of (-)^*} becomes isomorphisms in $\md \cX/\eff \cX$ and $\md \cX\op/\eff \cX\op$ if
\begin{align*}
& \Ext^1_{\md \cX\op}(F,\cX\op)\in \eff \cX \quad \text{for all } F\in \md \cX\op \\
& \Ext^1_{\md\cX}(F,\cX)\in \eff \cX\op \quad \text{for all } F\in \md \cX
\end{align*}
which ensures that $(-)^*\colon \md \cX/\eff \cX\xrightarrow{}\md \cX\op/\eff \cX\op$ is an equivalence. To capture this requirement, we introduce the following axiom:
\begin{enumerate}[label=(A\arabic*)]
\setcounter{enumi}{2}
\item\label{A3} Consider the following diagram
\begin{equation*}
\begin{tikzcd}
X_2 \arrow{rr}{f} \arrow{rd}{l} & & X_1 \arrow{r}{g} & X_0 \\
 & \arrow{ru}{h}X_2' & 
\end{tikzcd}
\end{equation*}
where $f$ is an arbitrary morphism in $\cX$, where $g$ is a weak cokernel of $f$, where $h$ is a weak kernel of $g$, and where $l$ is an induced map satisfying $h\circ l=f$ (which exists since $g\circ f=0$ and $h$ is a weak kernel of $g$). Then for any weak kernel $k\colon X_3'\to X_2'$ of $h$ the map $\begin{bmatrix}l & k\end{bmatrix}\colon X_2\oplus X_3'\to X_2'$ is an epimorphism. 
\end{enumerate}

\begin{Proposition}\label{Equivalent criteria for A3}
Assume $\cX$ is an additive category satsifying \ref{A1}, \ref{A1op}, \ref{A2}, \ref{A2op}. Then the following statements hold:
\begin{enumerate}
\item\label{Equivalent criteria for A3:1} $\cX$ satisfies \ref{A3} if and only if $\Ext^1_{\md \cX\op}(F,\cX)\in \eff \cX$ for all $F\in \md \cX\op$;
\item\label{Equivalent criteria for A3:2} $\cX$ satisfies \ref{A3}$\op$ if and only if $\Ext^1_{\md\cX}(F,\cX)\in \eff \cX\op$ for all $F\in \md \cX$.
\end{enumerate}
\end{Proposition}
\begin{proof}
We prove \ref{Equivalent criteria for A3:1}, \ref{Equivalent criteria for A3:2} is proved dually.  Let $f\colon X_2\to X_1$ be arbitrary, and choose $g$, $h$, $k$, $l$ as in \ref{A3}. Let $F$ be the cokernel of $-\circ f\colon \cX(X_1,-)\to \cX(X_2,-)$. Applying $(-)^*$ to the exact sequence
\[
\cX(X_0,-)\xrightarrow{-\circ g}\cX(X_1,-)\xrightarrow{-\circ f}\cX(X_2,-)\to F\to 0
\] 
we get a complex 
\begin{equation}\label{complex for A3}
F^*\to \cX(-,X_2)\xrightarrow{f\circ -}\cX(-,X_1)\xrightarrow{g\circ -}\cX(-,X_0)
\end{equation}
Let $K$ be the kernel of $g\circ -$. Since $g\circ h=0$, it follows that the map $h\circ -\colon \cX(-,X_2')\to \cX(-,X_1)$ factors through $K$ via a morphism $p\colon \cX(-,X_2')\to K$. Since $h$ is a weak kernel of $g$, it follows that any map $\cX(-,X)\to K$ with $X\in \cX$ must factor through $p$. Hence, $p$ is an epimorphism. Similarly, if we let $K'=\Ker (\cX(-,X_2')\xrightarrow{h\circ -}\cX(-,X_1))$, then since $k\colon X_3'\to X_2'$ is a weak kernel of $h$, it follows that $\cX(-,X_3')\xrightarrow{k\circ -}\cX(-,X_2')$ factors through $K'$ via an epimorphism $\cX(-,X'_3)\xrightarrow{q} K'$. Hence, we get a commutative diagram 
\begin{equation*}
\begin{tikzcd}[ampersand replacement =\&]
0\arrow{r}{} \& \cX(-,X_3') \arrow{r}{\begin{bmatrix}0 \\ 1\end{bmatrix}} \arrow{d}{q} \& \cX(-,X_2)\oplus \cX(-,X_3') \arrow{r}{\begin{bmatrix}1 & 0\end{bmatrix}} \arrow{d}{r} \& \cX(-,X_2) \arrow{r}{} \arrow{d}{} \& 0 \\
0\arrow{r} \& K' \arrow{r}{} \& \cX(-,X_2')\arrow{r}{p} \& K \arrow{r}{} \& 0 
\end{tikzcd}
\end{equation*}
with exact rows, where $r=\begin{bmatrix}
l\circ - & k\circ -\end{bmatrix}$. Note that the cokernel of $\cX(-,X_2)\to K$ is $\Ext^1_{\md \cX\op}(F,\cX)$. By the snake lemma, it follows that the cokernel of 
\[
r\colon \cX(-,X_2)\oplus \cX(-,X_3')\to \cX(-,X_2')
\]
is also $\Ext^1_{\md \cX\op}(F,\cX)$, since $q$ is an epimorphism. Hence, $\Ext^1_{\md \cX\op}(F,\cX)\in \eff \cX$ if and only if $r$ is an epimorphism in $\md \cX/\eff \cX$. Since a map in $\cX$ is an epimorphism in $\md \cX/\eff \cX$ if and only if it is an epimorphism in $\cX$, it follows that $r$ is an epimorphism in $\md \cX/\eff \cX$ if and only if $\begin{bmatrix}l & k\end{bmatrix}\colon X_2\oplus X_3'\to X_2'$ is an epimorphism in $\cX$. This proves the claim.
\end{proof}

\begin{Remark}\label{Remark A3 and A3'}
Note that the "only if" direction of the proof of Proposition \ref{Equivalent criteria for A3} uses the following alternative version of \ref{A3}, which therefore must be equivalent to \ref{A3} under the assumption of \ref{A1}, \ref{A1op}, \ref{A2}, \ref{A2op}:
\begin{enumerate}[label=(A\arabic*')]
\setcounter{enumi}{2}
\item\label{A3'} Let $f\colon X_2\to X_1$ be a morphism in $\cX$. Then there exists a diagram 
\begin{equation*}
\begin{tikzcd}
X_2 \arrow{rr}{f} \arrow{rd}{l} & & X_1 \arrow{r}{g} & X_0 \\
X_3'\arrow{r}{k} & \arrow{ru}{h}X_2' & 
\end{tikzcd}
\end{equation*}
where $g$ is a weak cokernel of $f$, where $h$ is a weak kernel of $g$, where $k$ is a weak kernel of $h$, where $l$ is a map satisfying $h\circ l=f$, and where the map $\begin{bmatrix}l & k\end{bmatrix}\colon X_2\oplus X_3'\to X_2'$ is an epimorphism. 
\end{enumerate}
\end{Remark}

We finish by giving a characterization of generating cogenerating functorially finite subcategories in terms of intrinsic axioms.

\begin{Theorem}\label{Embedding Theorem generating cogenerating functorially finite}
Let $\mathsf{P}$ be the axioms  \ref{A1}, \ref{A1op}, \ref{A2}, \ref{A2op}, \ref{A3} and \ref{A3}\op. Then $\mathsf{P}$ axiomatizes generating, cogenerating functorially finite subcategories of abelian categories.
\end{Theorem}

\begin{proof}
If $\cX$ satisfies \ref{A1}, \ref{A1op}, \ref{A2}, \ref{A2op}, \ref{A3} and \ref{A3}\op, then by Proposition \ref{Equivalent criteria for A3} and the exact sequences \eqref{unit of (-)^*} and \eqref{counit of (-)^*} it follows that $(-)^*$ induces an equivalence
\[
\md \cX/\eff \cX\cong (\md \cX\op/\eff \cX\op)\op
\]
which commutes with the natural inclusions 
\[
\cX\to \md \cX/\eff \cX\quad \text{and} \quad \cX\to (\md \cX\op/\eff \cX\op)\op.
\]
By Proposition \ref{essential image satisfy B1-B3} and its dual we get that $\cX$ is (equivalent to) a generating cogenerating subcategory of $\md \cX/\eff \cX$. Therefore $\Omega^2_{\cX}(\md \cX/\eff \cX)=\md \cX/\eff \cX$, and since $\cX$ satisfies \ref{B3} it follows that $\cX$ is contravariantly finite in $\md \cX/\eff \cX$. Similarly, $\cX$ is also covariantly finite in $\md \cX/\eff \cX$. This shows that $\cX$ is a generating cogenerating functorially finite subcategory. 

Now assume $\cX\subseteq \cA$ is a generating cogenerating functorially finite subcategory of an abelian category $\cA$. By Theorem \ref{Embedding Theorem} and its dual it follows that $\cX$ satisfies \ref{A1}, \ref{A1op}, \ref{A2} and \ref{A2op}. We only need to show that $\cX$ satisfies \ref{A3}, since the proof for \ref{A3}\op{ } is dual. Note first that a morphism $g\colon X_1\to X_0$ in $\cX$ is a weak cokernel of $f\colon X_2\to X_1$ if and only if the induced map $\Coker f\to X_0$ is a left $\cX$-approximation (and therefore also a monomorphism). The dual statement holds for weak kernels. Now assume we are given $l$,$f$,$g$,$h$,$k$ as in \ref{A3}. Since $\Ker g \cong \Ker (X_1\to \Coker f)$, it follows that $f$ factors through $\Ker g$ via an epimorphism $p\colon X_2\to \Ker g$. Also, since $h$ is a weak kernel of $g$ and $k$ is a weak kernel of $h$, the map $h$ factors through $\Ker g$ via an epimorphism $q\colon X_2'\to \Ker g$ and the map $k$ factors through $\Ker h$ via an epimorphism $p'\colon X_3'\to \Ker h$. Hence we get a commutative diagram with exact rows
\begin{equation*}
\begin{tikzcd}[ampersand replacement =\&]
0\arrow{r}{} \& X_3' \arrow{r}{\begin{bmatrix}0 \\ 1\end{bmatrix}} \arrow{d}{p'} \& X_2\oplus X_3' \arrow{r}{\begin{bmatrix}1 & 0\end{bmatrix}} \arrow{d}{\begin{bmatrix}l & k\end{bmatrix}} \& X_2 \arrow{r}{} \arrow{d}{p} \& 0 \\
0\arrow{r} \& \Ker h \arrow{r}{} \& X_2'\arrow{r}{q} \& \Ker g \arrow{r}{} \& 0 
\end{tikzcd}.
\end{equation*}
Since the leftmost and rightmost vertical map are epimorphism, the middle map must be an epimorphism. Therefore, $\cX$ satisfies \ref{A3}.
\end{proof}

\section{Rigid subcategories}

In this section we assume $\cX$ is a generating cogenerating functorially finite subcategory of an abelian category $\cA$. By Theorem \ref{Embedding Theorem generating cogenerating functorially finite} we know that this is equivalent to \ref{A1}, \ref{A1op}, \ref{A2}, \ref{A2op}, \ref{A3}, \ref{A3}\op. Now we want to determine the intrinsic axiom needed to capture the property
\begin{equation}\label{Ext vanishing}
\Ext^i_{\cA}(X,X')=0 \text{ for } 0< i< d \text{ and } X,X'\in \cX.
\end{equation}
We consider the following:
\begin{enumerate}[label=($d$-Rigid)]
\item\label{d-Rigid} For all epimorphism $f_1\colon X_1\to X_0$ in $\cX$ there exists a sequence
\[
X_{d+1}\xrightarrow{f_{d+1}}X_d\xrightarrow{f_d}\cdots \xrightarrow{f_3}X_2\xrightarrow{f_2}X_1\xrightarrow{f_1}X_0
\]
with $f_{i+1}$ a weak kernel of $f_i$ and $f_i$ a weak cokernel of $f_{i+1}$ for all $1\leq i\leq d$.
\end{enumerate}

The following theorem relates these two notions:

\begin{Theorem}\label{d-Rigid and Ext vanishing}
Let $\cX$ be a generating cogenerating functorially finite subcategory of an abelian category $\cA$. Then $\cX$ satisfies \ref{d-Rigid} if and only if
\[
\Ext^i_{\cA}(X,X')=0 
\] 
for all $0< i< d$ and all $ X,X'\in \cX$.
\end{Theorem}

\begin{Remark}
By Lemma \ref{Weak kernels and cokernels} it follows that \ref{d-Rigid} with $d=1$ is equivalent to \ref{A2} (under the assumption that weak kernels and cokernels exist). Hence, it holds automatically for a generating cogenerating functorially finite subcategory. This is reflected by the fact that condition \eqref{Ext vanishing} is empty for $d=1$.
\end{Remark}

\begin{proof}[Proof of "if" part of Theorem \ref{d-Rigid and Ext vanishing}]
Assume $f_1\colon X_1\to X_0$ is an epimorphism in $\cX$. Then $f_1$ must also be an epimorphism in $\cA$. Choose a right $\cX$-approximation $X_2\to \Ker f_1$, which must be an epimorphism since $\cX$ is generating. Let $f_2$ denote the composite $X_2\to \Ker f\to X_1$. We continue this construction iteratively for $1\leq i\leq d$, i.e. we choose a right $\cX$-approximation $X_{i+1}\to \Ker f_i$ and we let $f_{i+1}$ denote the composite $X_{i+1}\to \Ker f_i\to X_i$. Then we get an exact sequence
\[
X_{d+1}\xrightarrow{f_{d+1}}X_d\xrightarrow{f_d}\cdots \xrightarrow{f_3}X_2\xrightarrow{f_2}X_1\xrightarrow{f_1}X_0\to 0
\]
in $\cA$ where $f_{i+1}$ is a weak kernel of $f_i$ for all $1\leq i\leq d$. Applying $\cA(-,X')$ with $X'\in \cX$ and using that $\Ext^j_{\cA}(X_i,X')=0$ for all $0<j<d$  and $0\leq i\leq d+1$, we get an exact sequence
\[
0\to \cA(X_0,X')\xrightarrow{-\circ f_1}\cA(X_1,X')\xrightarrow{-\circ f_2}\cdots \xrightarrow{-\circ f_d}\cA(X_d,X')\xrightarrow{-\circ f_{d+1}}\cA(X_{d+1},X')
\] 
In particular, since the sequences $\cA(X_{i-1},X')\xrightarrow{-\circ f_i}\cA(X_i,X')\xrightarrow{-\circ f_{i+1}}\cA(X_{i+1},X')$ are exact for all $1\leq i\leq d$ and all $X'\in \cX$, it follows that $f_{i+1}$ is a weak cokernel of $f_i$ for all $1\leq i\leq d$. This proves the claim.
\end{proof}

The goal in the remaining part of the section is prove the converse.

\begin{Lemma}\label{1-Rigid}
Let $\cX$ be a generating cogenerating functorially finite subcategory of an abelian category $\cA$. Assume $\cX$ satisfies \ref{d-Rigid} for some $d>1$. Then $\Ext^1_{\cA}(X,X')=0$ for all $X,X'\in \cX$. 
\end{Lemma}
 
\begin{proof}
Let $0\to X'\xrightarrow{g}A\xrightarrow{f}X\to 0$ be an exact sequence in $\cA$ with $X,X'\in \cX$. Choose a right $\cX$-approximation $p\colon X_1\to A$. Then $g$ factors through $p$ via a monomorphism $i\colon X'\to X_1$. We therefore get a commutative diagram
\begin{equation*}
\begin{tikzcd}
0\arrow{r}{} & X' \arrow{r}{i} \arrow{d}{1} & X_1 \arrow{r}{} \arrow{d}{p} & \Coker i \arrow{r}{} \arrow{d}{} & 0\\
0\arrow{r} & X' \arrow{r}{g} & A \arrow{r}{f} & X \arrow{r}{} & 0 
\end{tikzcd}.
\end{equation*}
with exact rows. This gives an exact sequence
\[
0\to X'\oplus \Ker p\xrightarrow{\begin{bmatrix}i & j\end{bmatrix}} X_1\xrightarrow{f\circ p} X\to 0
\]
where $j\colon \Ker p\to X_1$ is the canonical monomorphism. By \ref{d-Rigid}, there exists a sequence
\[
X_3\xrightarrow{l}X_2\xrightarrow{k}X_1\xrightarrow{f\circ p}X\to 0
\] 
where $l$ is a weak kernel of $k$ and $k$ is a weak kernel of $f\circ p$, and where $k$ is weak cokernel of $l$ and $f\circ p$ is a weak cokernel of $k$, and where $X_3,X_2\in \cX$. Then 
\[
\Coker l\cong\im k\cong\Ker (f\circ p)=X'\oplus \Ker p.
\]
Since $k$ is a weak cokernel of $l$ it follows that the induced map
\[
\Coker l\cong X'\oplus \Ker p \xrightarrow{\begin{bmatrix}i & j\end{bmatrix}} X_1
\] 
is a left $\cX$-approximation. Hence $X'\oplus \Ker p \xrightarrow{\begin{bmatrix}1 & 0\end{bmatrix}}X'$ factors through $X_1$. This means that there exists a map $s\colon X_1\to X'$ such that $s\circ i=1_{X'}$ and $s\circ j=0$. Since $A\cong \Coker j$, we get an induced map $t\colon A\to X'$ satisfying $t\circ g=1_{X'}$. Therefore the sequence $0\to X'\xrightarrow{g}A\xrightarrow{f}X\to 0$ is split. Since the elements in $\Ext^1_{\cA}(X,X')$ can be described in terms of Yoneda extensions, this proves the claim.
\end{proof}

\begin{Lemma}\label{1 rigid not just in X}
Let $\cX$ be a generating cogenerating functorially finite subcategory of an abelian category $\cA$. Assume $\cX$ satisfies \ref{d-Rigid} for some $d>1$, and let $f\colon X_1\to X_0$ be a weak cokernel in $\cX$. Then $\Ext^1_{\cA}(\Coker f,X)=0$ for all $X\in \cX$. 
\end{Lemma}

\begin{proof}
Applying $\cA(-,X)$ for $X\in \cX$ to the exact sequence $0\to \im f\to X_0\to \Coker f\to 0$ gives an exact sequence
\[
0\to \cA(\Coker f,X)\to \cA(X_0,X)\to \cA(\im f, X)\to \Ext^1_{\cA}(\Coker f,X) \to \Ext^1_{\cA}(X_0,X)
\]
Now $\Ext^1_{\cA}(X_0,X)=0$ by Lemma \ref{1-Rigid}. Also, since $f$ is a weak cokernel, the map $\im f \to X_0$ is a left $\cX$-approximation. Therefore $\cA(X_0,X)\to \cA(\im f,X)$ is surjective. By considering the exact sequence above it follows that $\Ext^1_{\cA}(\Coker f,X)=0$. This proves the claim.
\end{proof}

\begin{proof}[Proof of "only if" part of Theorem \ref{d-Rigid and Ext vanishing}]
We prove that $\Ext^i_{\cA}(X',X)=0$ for all $X,X'\in \cX$ and $0<i<d$ by induction on $i$. For $i=1$ this follows from Lemma \ref{1-Rigid}. Assume $\Ext^i_{\cA}(X',X)=0$ for all $X,X'\in \cX$ and all $0<i\leq j$ with $j<d-1$. We prove that $\Ext^{j+1}_{\cA}(X',X)=0$ for all $X,X'\in \cX$. Let
\[
0\to X\xrightarrow{f_{j+2}}A_{j+1}\xrightarrow{f_{j+1}}\cdots \xrightarrow{f_2}A_1\xrightarrow{f_1}X'\to 0
\]
be an exact sequence with $X,X'\in \cX$. Choose an epimorphism $g\colon X_1\to A_1$ with $X_1\in \cX$, and let $g_1=f_1\circ g$. Next take the pullback square
\begin{equation*}
\begin{tikzcd}
 A_2' \arrow{r}{} \arrow{d}{} & \Ker g_1 \arrow{d}{} \\
 A_2 \arrow{r}{} & \Ker f_1  
\end{tikzcd}.
\end{equation*}
and let $g_2$ be the composite $A_2'\to \Ker g_1\to X_1$. Now constrict iteratively $A_k'$ and $g_k\colon A_k'\to A_{k-1}'$ for $k\leq j+1$ such that
\begin{equation*}
\begin{tikzcd}
 A_k' \arrow{r}{} \arrow{d}{} & \Ker g_{k-1} \arrow{d}{} \\
 A_k \arrow{r}{} & \Ker f_{k-1} 
\end{tikzcd}.
\end{equation*} 
is a pullback square and $g_k$ is the composite $A_k'\to \Ker g_{k-1}\to A_{k-1}'$ (note that $\Ker g_k\cong \Ker f_k$ for $k\geq 2$ and $A_k'\cong A_k$ for $k\geq 3$). Then we get a commutative diagram with exact rows
\begin{equation*}
\begin{tikzcd}
0\arrow{r}{} & X \arrow{r}{g_{j+2}} \arrow{d}{1} & A'_{j+1} \arrow{r}{g_{j+1}} \arrow{d}{} & \cdots \arrow{r}{g_3} & A_2' \arrow{r}{g_2} \arrow{d}{} & X_1 \arrow{r}{g_1} \arrow{d}{} & X'\arrow{r}{} \arrow{d}{1} & 0\\
0\arrow{r}{} & X \arrow{r}{f_{j+2}} & A_{j+1} \arrow{r}{f_{j+1}} & \cdots \arrow{r}{f_3} & A_2 \arrow{r}{f_2} & A_1 \arrow{r}{f_1} & X'\arrow{r}{}  & 0 
\end{tikzcd}.
\end{equation*}
Hence, both exact sequences represents the same element in the Yoneda Ext-group \\ $\Ext^{j+1}_{\cA}(X',X)$. Therefore, it is sufficient to show that the upper exact sequence is $0$ as an element in  $\Ext^{j+1}_{\cA}(X',X)$. For this, it suffices to show that $\Ext^j_{\cA}(\Ker g_1,X)=0$. Now by axiom \ref{d-Rigid} there exists an exact sequence
\[
X_{j+3}\xrightarrow{h_{j+3}}\cdots \xrightarrow{h_4}X_3\xrightarrow{h_3}X_2\xrightarrow{h_2}X_1\xrightarrow{h_1}X'\to 0
\]
where $g_1=h_1$, and where $h_{i+1}$ is a weak kernel of $h_i$ and $h_i$ is a weak cokernel of $h_{i+1}$ for $1\leq i\leq j+2$, and where $X_i\in \cX$ for $1\leq i\leq j+3$. Now consider the exact sequence
\[
0\to \Ker h_{i+1}\to X_{i+1}\to \Ker h_i\to 0
\]
where $1\leq i \leq j-1$. Applying $\cA(-,X)$, we get exact sequences
\begin{multline*}
0\to \cA(\Ker h_i,X)\to \cA(X_{i+1},X)\to \cA(\Ker h_{i+1},X)\to \cdots \to \Ext^{j-i}_{\cA}(X_{i+1},X) \\
 \to \Ext^{j-i}_{\cA}(\Ker h_{i+1},X)\to \Ext^{1+j-i}_{\cA}(\Ker h_i,X)\to \Ext^{1+j-i}_{\cA}(X_{i+1},X)\to \cdots
\end{multline*}
Since $\Ext^{j-i}_{\cA}(X_{i+1},X)=0=\Ext^{1+j-i}_{\cA}(X_{i+1},X)$ for $1\leq i\leq j-1$ by the induction hypothesis, we get that 
\[
\Ext^{1+j-i}_{\cA}(\Ker h_i,X)\cong \Ext^{j-i}_{\cA}(\Ker h_{i+1},X)
\]
Hence
\[
\Ext^j_{\cA}(\Ker h_1,X)\cong \Ext^{j-1}_{\cA}(\Ker h_2,X)\cong \cdots \cong \Ext^1_{\cA}(\Ker h_{j},X)
\]
Since $\Ker h_j\cong \Coker h_{j+2}$ and $h_{j+2}$ is a weak cokernel, it follows that 
\[
\Ext^1_{\cA}(\Coker h_{j+2},X)=0
\] by Lemma \ref{1 rigid not just in X}.
Hence $\Ext^j_{\cA}(\Ker h_1,X)=0$, which proves the claim.
\end{proof}

\section{d-abelian categories are d-cluster tilting}

In this section we show that any $d$-abelian category is equivalent to a $d$-cluster tilting subcategory of an abelian category. More precisely, we show that being a $d$-abelian is equivalent to having $d$-kernels and $d$-cokernels and satisfying axioms \ref{A1}, \ref{A1op}, \ref{A2}, \ref{A2op}, \ref{A3}, \ref{A3}\op, \ref{d-Rigid}, and we show that such categories axiomatizes $d$-cluster tilting subcategories.  

We first recall the definition of a $d$-cluster tilting subcategory:

\begin{Definition}\label{Definition cluster tilting subcategory}
Let $\cX$ be a full subcategory of an abelian category $\cA$, and let $d>0$ be a positive integer. We say that $\cX$ is $d$\emphbf{-cluster tilting} in $\cA$ if the following hold:
\begin{enumerate}
\item\label{Definition cluster tilting subcategory:1} $\cX$ is a generating cogenerating functorially finite subcategory of $\cA$;
\item\label{Definition cluster tilting subcategory:2} We have
\begin{align*}
\cX & = \{A\in \cA \mid \Ext^i_{\cA}(A,X)=0 \text{ for }1\leq i\leq d-1 \text{ and }X\in \cX\} \\ 
& = \{A\in \cA \mid \Ext^i_{\cA}(X,A)=0 \text{ for }1\leq i\leq d-1 \text{ and }X\in \cX\};
\end{align*}
\end{enumerate}
\end{Definition}

We need the following result on $d$-cluster tilting subcategories

\begin{Lemma}\label{result on d-cluster tilting}
Let $\cX$ be a generating cogenerating functorially finite subcategory of an abelian category $\cA$. Assume $\Ext^i_{\cA}(X,X')=0$ for $1\leq i\leq d-1$ and $X,X'\in \cX$. The following are equivalent:
\begin{enumerate}
\item\label{result on d-cluster tilting:1} $\cX$ is $d$-cluster tilting in $\cA$;
\item\label{result on d-cluster tilting:2} $\cX$ is closed under direct summands, and for any $A\in \cA$ there exists exact sequences
\[
0\to A\to X_{-1}\to \cdots \to X_{-d}\to 0 \quad \text{and} \quad 0\to X_d'\to \cdots \to X'_1\to A\to 0
\]
where $X_i,X_i'\in \cX$ for $1\leq i\leq d$.
\end{enumerate}
\end{Lemma}

\begin{proof}
This follows from \cite[Proposition 2.2.2]{Iya07a}.
\end{proof}

Let $\cX$ be an additive category, and let $f\colon X_1\to X_0$ be a morphism in $\cX$. Following \cite{Jas16}, we say that a sequence
\[
X_{d+1}\to X_d \to \cdots \to X_1
\]
in $\cX$ is a $d$\emphbf{-kernel} of $f$ if the sequence of abelian groups
\[
0\to \cX(X,X_{d+1})\to \cX(X,X_d)\to \cdots \to \cX(X,X_1)\xrightarrow{f\circ -} \cX(X,X_0)
\]
is exact for all $X\in \cX$. Dually, the sequence
\[
X_0\to X_{-1}\to \cdots \to X_{-d}
\]
is a $d$\emphbf{-cokernel} of $f$ if the sequence of abelian groups
\[
0\to \cX(X_{-d},X)\xrightarrow{} \cX(X_{-d+1},X)\to \cdots \to \cX(X_{0},X)\xrightarrow{-\circ f} \cX(X_1,X)
\]
is exact for all $X\in \cX$. 

\begin{Theorem}\label{axioms implying d-cluster tilting}
Let $\cX$ be an idempotent complete additive category satisfying \ref{A1}, \ref{A1op}, \ref{A2}, \ref{A2op}, \ref{A3}, \ref{A3}\op, and \ref{d-Rigid}. Assume furthermore that every morphism in $\cX$ has a $d$-kernel and a $d$-cokernel. Then $\cX$ is equivalent to a $d$-cluster tilting subcategory of an abelian category.
\end{Theorem}

\begin{proof}
By Theorem \ref{Embedding Theorem generating cogenerating functorially finite} and Theorem \ref{d-Rigid and Ext vanishing} we can assume $\cX$ is a generating cogenerating functorially finite subcategory of an abelian category $\cA$ satisfying $\Ext^i_{\cA}(X,X')=0$ for all $0<i<d$ and $X,X'\in \cX$. In particular, since $\cX$ is generating and cogenerating, for any object $A\in \cA$ there exists morphisms $f\colon X_{1}\to X_0$ and $g\colon X_0'\to X_{-1}'$ in $\cX$ with $\Coker f\cong A$ and $\Ker g\cong A$. Since $\cX$ is idempotent complete, taking the $d$-kernel of $f$ and the $d$-cokernel of $g$, we see that condition \ref{result on d-cluster tilting:2} in Lemma \ref{result on d-cluster tilting} holds. Hence $\cX$ must be $d$-cluster tilting in $\cA$.
\end{proof}

Next we recall the definition of $d$-abelian categories. Following \cite{Jas16}, we say that a complex
\[
X_{d+1}\xrightarrow{f_{d+1}} \cdots \to X_1\xrightarrow{f_1} X_0
\] 
is $d$\emphbf{-exact} if $X_{d+1}\xrightarrow{f_{d+1}} \cdots \xrightarrow{f_2} X_1$ is a $d$-kernel of $f_1$ and $X_{d}\xrightarrow{f_{d}} \cdots \to X_1\xrightarrow{f_1} X_0$ is a $d$-cokernel of $f_{d+1}$.

\begin{Definition}[Definition 3.1 in \cite{Jas16}]\label{d-abelian categories}
Let $d$ be a positive integer, and let $\cX$ be an additive category. We say that $\cX$ is a $d$-abelian category if it satisfies the following
\begin{enumerate}
\item\label{d-abelian categories:1} $\cX$ is idempotent complete;
\item\label{d-abelian categories:2} Every morphism in $\cX$ has a $d$-kernel and a $d$-cokernel;
\item\label{d-abelian categories:3} Any complex 
\[
X_{d+1}\xrightarrow{f_{d+1}} \cdots \to X_1\xrightarrow{f_1} X_0
\] 
where $f_1$ is an epimorphism and $X_{d+1}\xrightarrow{f_{d+1}} \cdots \xrightarrow{f_2} X_1$ is a $d$-kernel of $f$ must be $d$-exact;
\item\label{d-abelian categories:4} Any complex 
\[
X_{d+1}\xrightarrow{f_{d+1}} \cdots \to X_1\xrightarrow{f_1} X_0
\] 
where $f_{d+1}$ is a monomorphism and $X_{d}\xrightarrow{f_{d}} \cdots \to X_1\xrightarrow{f_1} X_0$ is a $d$-cokernel of $f_{d+1}$ must be $d$-exact.
\end{enumerate}
\end{Definition}

\begin{Proposition}\label{d-abelian satisfies axioms}
Let $\cX$ be an additive category.  Then $\cX$ is $d$-abelian if and only if it is idempotent complete, has $d$-kernels and cokernels, and satisfies \ref{A1}, \ref{A1op}, \ref{A2}, \ref{A2op}, \ref{A3}, \ref{A3}\op and \ref{d-Rigid}.
\end{Proposition}

\begin{proof}
The "if" direction follows from Theorem \ref{axioms implying d-cluster tilting} and the fact that any $d$-cluster tilting subcategory is $d$-abelian by \cite[Theorem 3.16]{Jas16}. Conversely, assume $\cX$ is a $d$-abelian category. Since $\cX$ has $d$-kernels and $d$-cokernels, axioms \ref{A1}, \ref{A1op} hold automatically. Also by Definition \ref{d-abelian categories} \ref{d-abelian categories:3} and the fact that $\cX$ has $d$-kernels it follows that \ref{A2} and \ref{d-Rigid} hold, and dually by Definition \ref{d-abelian categories} \ref{d-abelian categories:4} and the fact that $\cX$ has $d$-cokernels it follows that \ref{A2}\op \text{ }holds. We show that $\cX$ satisfies \ref{A3}. Let $f^0\colon X^0\to X^1$ be a morphism, and let $f^1\colon X^1\to X^2$ be a weak cokernel of $f^0$. Then by \cite[Proposition 3.13]{Jas16} there exists objects $Y^1_1$ and $Y^2_1$ in $\cX$ and morphisms $g^1_1\colon Y^1_1\to X^1$ and $g^2_1\colon Y^2_1\to Y^1_1$ and $p^0_0\colon X^0\to Y^1_1$ such that
\begin{enumerate}
\item $g^1_1$ is a weak kernel of $f^1$, $g_1^2$ is a weak kernel of $g^1_1$;
\item $g^1_1\circ p^0_0=f^0$;
\item The map $\begin{bmatrix}p^0_0 & g^2_1\end{bmatrix}\colon X^0\oplus Y^2_1\to Y^1_1$ is an epimorphism.
\end{enumerate}
This shows that $\cX$ satisfies axiom \ref{A3'}, which by Remark \ref{Remark A3 and A3'} is equivalent to \ref{A3}. Axiom \ref{A3}\op  \text{ }is proved dually.
\end{proof}

\section{Precluster tilting subcategories}
Let $\cX$ be an additive category satisfying \ref{A1}, \ref{A1op}, \ref{A2}, \ref{A2op}, \ref{A3}, \ref{A3}\op \text{ }and \ref{d-Rigid}. The goal in this section is to find additional axioms on $\cX$ so that it gives an axiomatization of $d$-precluster tilting subcategories as introduced in \cite{IS18}.  In order to do this, we need to reformulate the definition of precluster tilting subcategories so that it makes sense for any abelian category.

 In the following we fix a commutative artinian ring $R$, an Artin $R$-algebra $\Lambda$, and we let $\md \Lambda$ be the category of finitely generated (right) $\Lambda$-modules. We denote by $\underline{\md}\Lambda$ and $\overline{\md} \Lambda$ the quotients of $\md \Lambda$ by the ideals of morphisms factoring through a projective or injective object, respectively, and $\Omega\colon \underline{\md}\Lambda\to \underline{\md}\Lambda$ and $\Omega^-\colon \overline{\md}\Lambda\to \overline{\md}\Lambda$ the syzygy and cosyzygy functor,  respectively. The objects in $\md \Lambda$ and their image in $\underline{\md}\Lambda$ and $\overline{\md} \Lambda$ will be denoted by the same letter. For a subcategory $\cX$ of $\md \Lambda$, we let $\underline{\cX}$ and $\overline{\cX}$ denote the smallest subcategories of  $\underline{\md}\Lambda$ and $\overline{\md} \Lambda$ which are closed under isomorphisms and contain all $X\in \cX$.
 
To define $d$-precluster tilting subcategories, we consider the $d$-Auslander-Reiten translations defined by
\[
\tau_d := \tau \circ \Omega^{d-1}\colon \underline{\md}\Lambda \to \overline{\md} \Lambda \quad \text{and} \quad \tau^-_d := \tau^- \circ \Omega^{-(d-1)}\colon \overline{\md}\Lambda \to \underline{\md} \Lambda
\] 
where $\tau\colon \underline{\md}\Lambda \to \overline{\md}\Lambda$ denotes the classical Auslander-Reiten translation with quasi-inverse $\tau^-\colon \overline{\md}\Lambda \to \underline{\md}\Lambda$. 

\begin{Definition}[Definition 3.2 in \cite{IS18}]\label{precluster tilting}
Let $\cX$ be an additive subcategory of $\md \Lambda$. Assume $\cX$ is closed under direct summands. We say that $\cX$ is a $d$\emphbf{-precluster tilting subcategory} if it satisfies the following:
\begin{enumerate}
\item\label{precluster tilting:1} $\cX$ is a generating cogenerating subcategory of $\md \Lambda$;
\item\label{precluster tilting:2} $\tau_d(X)\in \overline{\cX}$ and $\tau^-_d(X)\in \underline{\cX}$ for all $X\in \cX$;
\item\label{precluster tilting:3} $\Ext^i_{\Lambda}(X,X')=0$ for all $X,X'\in \cX$ and $0<i<d$;
\item\label{precluster tilting:4} $\cX$ is a functorially finite subcategory of $\md \Lambda$.
\end{enumerate}
\end{Definition}

The appearance of $\tau_d$ and $\tau^-_d$ makes precluster tilting subcategories difficult to axiomatize. Luckily, criterion \ref{precluster tilting:2} in Definition \ref{precluster tilting} can be reformulated in homological terms. Our first goal is to do this. For $d>1$ such a reformulation is already known, as the following result shows. For simplicity we set
\[
{}^{\perp_d}\cX:=\{M\in \md \Lambda\mid \Ext^i_{\Lambda}(M,X)=0 \text{ for all } 0<i<d \text{ and }X\in \cX\}
\]
and 
\[
\cX^{\perp_d}:=\{M\in \md \Lambda\mid \Ext^i_{\Lambda}(X,M)=0 \text{ for all } 0<i<d \text{ and }X\in \cX\}.
\]

\begin{Lemma}[Proposition 3.8 part b) in \cite{IS18}]\label{reformulation of precluster tilting d>1}
Let $d>1$ be an integer. Assume $\cX$ is an  additive subcategory of $\md \Lambda$ closed under direct summands and satisfying \ref{precluster tilting:1}, \ref{precluster tilting:3}, \ref{precluster tilting:4} in Definition \ref{precluster tilting}. Then the following are equivalent:
\begin{enumerate}
\item $\cX$ is $d$-precluster tilting;
\item ${}^{\perp_d}\cX=\cX^{\perp_d}$.
\end{enumerate}
\end{Lemma}

We also need the following lemma which gives a simpler criterion for when ${}^{\perp_d}\cX=\cX^{\perp_d}$.

\begin{Lemma}\label{useful lemma precluster tilting}
Let $d>1$ be a positive integer, and assume $\cX$ is an additive subcategory of $\md \Lambda$ closed under direct summands and satisfying \ref{precluster tilting:1}, \ref{precluster tilting:3}, \ref{precluster tilting:4} in Definition \ref{precluster tilting}. Assume furthermore that ${\cX}^{\perp_d}\subseteq {}^{\perp_2}\cX$ and ${}^{\perp_d}\cX\subseteq\cX^{\perp_2}$. Then ${}^{\perp_d}\cX=\cX^{\perp_d}$.
\end{Lemma}

\begin{proof}
We prove by induction on $2\leq i\leq d$ that $\cX^{\perp_d}\subseteq {}^{\perp_i}\cX$. For $i=2$ this follows by assumption. Assume the claim holds for $2\leq i<d$, and we want to show that it holds for $i+1$. Let $M\in \cX^{\perp_d}$, choose a right $\cX$-approximation $f\colon X\to M$, and let $M'=\Ker f$. Applying $\Hom_{\Lambda}(X',-)$ with $X'\in \cX$ to the exact sequence 
\begin{equation}\label{equation reference}
0\to M'\to X\xrightarrow{f}M\to 0
\end{equation} gives a long exact sequence
\begin{multline*}
0\to \operatorname{Hom}_{\Lambda}(X',M')\to \operatorname{Hom}_{\Lambda}(X',X)\xrightarrow{f\circ -} \operatorname{Hom}_{\Lambda}(X',M)\to \Ext^1_{\Lambda}(X',M') \\
 \to \Ext^1_{\Lambda}(X',X) \to \cdots \to \Ext^{j-1}_{\Lambda}(X',M)\to \Ext^j_{\Lambda}(X',M') 
 \to \Ext^j_{\Lambda}(X',X)\to \cdots
\end{multline*}
Since $f$ is a right $\cX$-approximation, it follows that the map 
\[
\operatorname{Hom}_{\Lambda}(X',X)\xrightarrow{f\circ -} \operatorname{Hom}_{\Lambda}(X',M)
\]
is an epimorphism. Also, since $\Ext^j_{\Lambda}(X',X)=0$ for $0<j<d$ by Definition \ref{precluster tilting} \ref{precluster tilting:3} and $\Ext^j_{\Lambda}(X',M)=0$ for $0<j<d$ by assumption, it follows that $\Ext^j_{\Lambda}(X',M')=0$ for $0<j<d$. Hence $M'\in \cX^{\perp_d}$, and therefore $M'\in {}^{\perp_i}\cX$ by induction hypothesis. Now applying $\Hom_{\Lambda}(-,X')$ to \eqref{equation reference} and considering the long exact sequence we get
\[
\Ext^{i}_{\Lambda}(M,X')\cong \Ext^{i-1}_{\Lambda}(M',X')=0
\]
since $\Ext^i_{\Lambda}(X,X')=0=\Ext^{i-1}_{\Lambda}(X,X')$. This shows that $M\in {}^{\perp_{i+1}}\cX$. Therefore, by induction we get that $\cX^{\perp_d}\subseteq {}^{\perp_d}\cX$. The inclusion ${}^{\perp_d}\cX\subseteq\cX^{\perp_d}$ is proved dually. Combining the inclusions, we get that ${}^{\perp_d}\cX=\cX^{\perp_d}$, which proves the claim.
\end{proof}

Note that Lemma \ref{reformulation of precluster tilting d>1} only holds when $d>1$, so we still need a homological reformulation of  Definition \ref{precluster tilting} \ref{precluster tilting:2} when $d=1$. This is done by the following result.

\begin{Theorem}\label{reformulation precluster tilting}
Let $d$ be a positive integer, and assume $\cX$ is an additive subcategory of $\md \Lambda$ closed under direct summands and satisfying \ref{precluster tilting:1}, \ref{precluster tilting:3} and \ref{precluster tilting:4} in Definition \ref{precluster tilting}. The following are equivalent:
\begin{enumerate}
\item\label{reformulation precluster tilting:1} $\cX$ is a $d$-precluster tilting subcategory;
\item\label{reformulation precluster tilting:2} For any exact sequence in $\md \Lambda$
\[
0\to M'\xrightarrow{f_{d+1}} X_d\xrightarrow{f_d} \cdots \xrightarrow{f_3} X_2\xrightarrow{f_2} X_1\xrightarrow{f_1} M\to 0
\]
with $X_i\in \cX$ for $1\leq i\leq d$, the following hold:
\begin{enumerate}
\item If the induced map $X_i\to \im f_i$ is a right $\cX$-approximation for all $1\leq i\leq d$, then $f_{d+1}\colon M'\to X_d$ is a left $\cX$-approximation;
\item If the induced map $\im f_i\to X_{i-1}$ is a left $\cX$-approximation for all $2\leq i\leq d+1$, then $f_1\colon X_1\to M$ is a right $\cX$-approximation.
\end{enumerate}
\end{enumerate}
\end{Theorem} 

\begin{proof}
We prove the cases $d>1$ and $d=1$ separately. First assume $d>1$ and that $\cX$ is a $d$-precluster tilting subcategory. Let
\[
0\to M'\xrightarrow{f_{d+1}} X_d\xrightarrow{f_d} \cdots \xrightarrow{f_3} X_2\xrightarrow{f_2} X_1\xrightarrow{f_1} M\to 0
\]
be an exact sequences such that $X_i\in \cX$ and $X_i\to \im f_i$ is a right $\cX$-approximation for all $1\leq i\leq d$. Applying $\Hom_{\Lambda}(X,-)$ with $X\in \cX$ to the exact sequence
\[
0\to \im f_{i+1}\to X_{i}\to \im f_{i}\to 0
\] 
and using that $\Ext^j_{\Lambda}(X,X_{i})=0$ for $1\leq j<d$, we get that 
\[
\Ext^1_{\Lambda}(X,\im f_{i+1})=0 \quad \text{and} \quad \Ext^{j}_{\Lambda}(X,\im f_{i+1})\cong \Ext^{j-1}_{\Lambda}(X,\im f_{i})
\]
for $1\leq i\leq d$ and $2\leq j<d$. Hence, we have that
\[
\Ext^j_{\Lambda}(X,\im f_d)\cong \Ext^{j-1}_{\Lambda}(X,\im f_{d-1})\cong \cdots \cong \Ext^1_{\Lambda}(X,\im f_{d-j+1})=0
\]
for $0<j<d$. This shows that $\im f_d\in \cX^{\perp_d}$, so $\im f_{d}\in {}^{\perp_2}\cX$ by Lemma \ref{reformulation of precluster tilting d>1}, and therefore $f_{d+1}\colon M'\to X_d$ is a left $\cX$-approximation. This together with the dual argument shows the implication \ref{reformulation precluster tilting:1} $\implies$ \ref{reformulation precluster tilting:2} for $d>1$.

Conversely, assume $d>1$ and that $\cX$ satisfies  part \ref{reformulation precluster tilting:2} of the theorem. Let $M'\in \cX^{\perp_d}$, and choose a right $\cX$-approximation $X_d\to M'$ and an exact sequence
\[
0\to M'\xrightarrow{} X_{d-1}\xrightarrow{f_{d-1}}X_{d-2}\xrightarrow{f_{d-2}}\cdots \xrightarrow{f_2}X_1\xrightarrow{f_1}M\to 0
\]
Since $\Ext^i_{\Lambda}(X,M')=0$ and $\Ext^i_{\Lambda}(X,X')=0$ for $0<i<d$ and $X,X'\in \cX$, it follows that the sequence
\[
\Hom_{\Lambda}(X,X_{d-1})\xrightarrow{f_{d-1}\circ -} \cdots \xrightarrow{f_2\circ -} \Hom_{\Lambda}(X,X_1) \xrightarrow{f_1\circ -} \Hom_{\Lambda}(X,M)\to 0
\]
is exact for $X\in \cX$. Hence the canonical map $X_i\to \im f_i$ is a right $\cX$-approximation for $1\leq i\leq d-1$. Therefore, by assumption we get that the map $\Ker f_d\to X_d$ is a left $\cX$-approximation, where $f_d$ is the composite $X_d\to M'\to X_{d-1}$. Hence, applying $\Hom_{\Lambda}(-,X)$ with $X\in \cX$ to the exact sequence
\[
0\to \Ker f_d\to X_d\to M'\to 0
\]
we get that $\Ext^1_{\Lambda}(M',X)=0$ so $M'\in {}^{\perp_2}\cX$. Since $M'\in \cX^{\perp_d}$ was arbitrary, this shows that $\cX^{\perp_d}\subseteq {}^{\perp_2}\cX$. The inclusion ${}^{\perp_d}\cX\subseteq \cX^{\perp_2}$ is proved dually, and the fact that $\cX$ is $d$-precluster tilting follows from Lemma \ref{reformulation of precluster tilting d>1} and Lemma \ref{useful lemma precluster tilting}.

Now we assume $d=1$ and $\cX$ is a $1$-precluster tilting subcategory. Let 
\[
0\to M\xrightarrow{f}X\xrightarrow{g}M'\to 0
\]
 be an exact sequence with $f$ a left $\cX$-approximation. Then since $\tau(X')\in \overline{\cX}$ for all $X'\in \underline{\cX}$, all morphisms $M\to \tau(X')$ will factor through $f$. Hence by \cite[Chapter IV, Corollary 4.4]{ARS95} all morphisms $X'\to M'$ with $X'\in \cX$ will factor through $g$. Therefore $g$ is a right $\cX$-approximation. Together with the dual argument this shows that $1$-precluster tilting implies condition \ref{reformulation precluster tilting:2} in the theorem. 
 
Finally, assume $d=1$ and that condition \ref{reformulation precluster tilting:2} holds for $\cX$. Assume furthermore that there exists an indecomposable module $X\in \cX$ for which $Y=\tau(X)\notin \overline{\cX}$. By abuse of notation we let $\tau(X)$ denote the indecomposable module in $\md \Lambda$ corresponding $Y$ in $\overline{\md}\Lambda$. Consider the commutative diagram with exact rows
\begin{equation*}
\begin{tikzcd}
0\arrow{r}{} & \tau(X) \arrow{r}{} \arrow{d}{1} & E \arrow{r}{} \arrow[dashed]{d}{} & X \arrow{r}{} \arrow[dashed]{d}{} & 0\\
0\arrow{r} & \tau(X) \arrow{r}{} & X'\arrow{r}{} & M \arrow{r}{} & 0 
\end{tikzcd}
\end{equation*}
where the top row is an almost split sequence, and where $\tau(X)\to X'$ is a left $\cX$-approximation with cokernel $M$. Since $\tau(X)\notin \cX$, the morphism $\tau(X)\to X'$ is not a split monomorphism, and therefore it factors through $\tau(X)\to E$. Hence we obtain vertical maps $E\to X'$ and $X\to M$ making the diagram commute. By assumption we have that $X'\to M$ is a right $\cX$-approximation, and hence $X\to M$ factors through $X'\to M$. Since the rightmost square is a pushout square, it follows that $E\to X$ is a split epimorphism. Therefore the sequence $0\to \tau(X)\to E\to X\to 0$ must be split, which is a contradiction. This shows that $\tau(X)\in \cX$. The implication $X\in \overline{\cX}\implies \tau^-(X)\in \underline{\cX}$ is proved dually.
\end{proof}

Motivated by this, we define $d$-precluster tilting subcategories for arbitrary abelian categories. By Theorem \ref{reformulation precluster tilting} it coincides with the classical definition for $\cA=\md \Lambda$.

\begin{Definition}\label{precluster tilting abelian}
Let $\cX$ be an additive subcategory of an abelian category $\cA$. Assume $\cX$ is closed under direct summands. We say that $\cX$ is a $d$\emphbf{-precluster tilting subcategory} if it satisfies the following:
\begin{enumerate}
\item\label{precluster tilting abelian:1} $\cX$ is a generating cogenerating subcategory of $\cA$;
\item\label{precluster tilting abelian:2} For any exact sequence in $\cA$
\[
0\to A'\xrightarrow{f_{d+1}} X_d\xrightarrow{f_d} \cdots \xrightarrow{f_3} X_2\xrightarrow{f_2} X_1\xrightarrow{f_1} A\to 0
\]
with $X_i\in \cX$ for $1\leq i\leq d$, the following hold:
\begin{enumerate}
\item If the induced map $X_i\to \im f_i$ is a right $\cX$-approximation for all $1\leq i\leq d$, then $f_{d+1}\colon A'\to X_d$ is a left $\cX$-approximation;
\item If the induced map $\im f_i\to X_{i-1}$ is a left $\cX$-approximation for all $2\leq i\leq d+1$, then $f_1\colon X_1\to A$ is a right $\cX$-approximation.
\end{enumerate}
\item\label{precluster tilting abelian:3} $\Ext^i_{\cA}(X,X')=0$ for all $X,X'\in \cX$ and $0<i<d$;
\item\label{precluster tilting abelian:4} $\cX$ is a functorially finite subcategory of $\cA$.
\end{enumerate}
\end{Definition}

We now introduce the necessary axiom to capture Definition \ref{precluster tilting abelian} \ref{precluster tilting abelian:2}.

\begin{enumerate}[label=(A4.$d$)]
\item\label{A4} Consider a sequence
\[
X_{d+1}\xrightarrow{f_{d+1}}X_d\xrightarrow{f_d}\cdots \xrightarrow{f_3}X_2\xrightarrow{f_2}X_1\xrightarrow{f_1}X_0\xrightarrow{f_0}X_{-1}
\]
with $f_{i+1}$ a weak kernel of $f_i$ for all $0\leq i\leq d$. Then $f_{d+1}$ is a weak cokernel.
\end{enumerate}

\begin{Theorem}\label{axiomatization precluster tilting}
Let $\cX$ be a generating cogenerating functorially finite subcategory of an abelian category $\cA$ satisfying $\Ext^i_{\cA}(X,X')=0$ for all $X,X'\in \cX$ and $0<i<d$. Then $\cX$ satisfies \ref{A4} and \ref{A4}\op\text{ } if and only if it is $d$-precluster tilting.
\end{Theorem}

\begin{proof}
This follows immediately from the fact that a  map $X\xrightarrow{f}X'$ is a weak kernel or weak cokernel if and only if the projection $X\to \im f$ is a right $\cX$-approximation or the inclusion $\im f\to X'$ is a left $\cX$-approximation, respectively.
\end{proof}

\bibliography{Mybibtex}
\bibliographystyle{plain} 

\end{document}